\documentclass[11pt,a4paper]{amsart}
\usepackage{amssymb,amsmath,amsthm,amsfonts,amscd,euscript,hyperref,fourier,bm}
\usepackage[in]{fullpage}
\numberwithin{equation}{section}
\newtheorem{thm}{Theorem}[section]
\newtheorem*{ithm}{Theorem}
\newtheorem{prop}[thm]{Proposition}
\newtheorem{lem}[thm]{Lemma}
\newtheorem{cor}[thm]{Corollary}
\theoremstyle{remark}
\newtheorem{rem}[thm]{Remark}
\newtheorem{definition}[thm]{Definition}
\newtheorem{example}[thm]{Example}

\newtheorem{conj}[thm]{Conjecture}

\newcommand{\nc}{\newcommand}
\nc{\cO}{\mathcal O}
\nc{\cF}{\mathcal F}
\nc{\cL}{\mathcal L}
\nc{\msl}{\mathfrak{sl}}
\nc{\mgl}{\mathfrak{gl}}
\nc{\U}{\mathrm U}
\nc{\bH}{\EuScript H}
\nc{\Res}{\mathrm{Res\ }}
\nc{\Lie}{\mathrm{Lie\ }}

\newcommand{\bZ}{{\mathbb Z}}

\newcommand{\fg}{{\mathfrak g}}

\newcommand{\bd}{{\mathbf d}}
\newcommand{\bw}{{\mathbf w}}
\newcommand{\M}{{\mathcal M}}
\nc{\ch}{\mathrm{ch}}
\nc{\la}{\lambda}
\nc{\msp}{\mathfrak{sp}}
\nc{\cd}{\cdots}
\nc{\hk}{\hookrightarrow}
\nc{\T}{\otimes}
\nc{\al}{\alpha}
\nc{\om}{\omega}
\nc{\veps}{\varepsilon}
\nc{\ket}{\rangle}

\begin{document}

\title{Koszul algebras and Donaldson-Thomas invariants}
\author{Vladimir Dotsenko}
\address{Institut de Recherche Math\'ematique Avanc\'ee, UMR 7501, Universit\'e de Strasbourg et CNRS, 7 rue Ren\'e-Descartes, 67000 Strasbourg, France}
\address{Institut Universitaire de France}
\email{vdotsenko@unistra.fr}

\author{Evgeny Feigin}
\address{HSE University, Faculty of Mathematics, Ulitsa Usacheva 6, Moscow 119048, Russia}  
\address{Skolkovo Institute of Science and Technology, Center for Advanced Studies, Bolshoy Boulevard 30, bld. 1, Moscow 121205, Russia}
\email{evgfeig@gmail.com}

\author{Markus Reineke}
\address{Ruhr-University Bochum, Faculty of Mathematics, Universit\"atstrasse 150, 44780 Bochum, Germany}
\email{Markus.Reineke@ruhr-uni-bochum.de}

\date{}

\begin{abstract}
For a given symmetric quiver $Q$, we define a supercommutative quadratic algebra $\mathcal{A}_Q$ whose Poincar\'e series is related to the motivic generating function of $Q$ by a simple change of variables. The Koszul duality between supercommutative algebras and Lie superalgebras assigns to the algebra $\mathcal{A}_Q$ its Koszul dual Lie superalgebra $\mathfrak{g}_Q$. We prove that the motivic Donaldson-Thomas invariants of the quiver $Q$ may be computed using the Poincar\'e series of a certain Lie subalgebra of $\mathfrak{g}_Q$ that can be described, using an action of the first Weyl algebra on $\mathfrak{g}_Q$, as the kernel of the operator $\partial_t$. This gives a new proof of positivity for motivic Donaldson--Thomas invariants. In addition, we prove that the algebra $\mathcal{A}_Q$ is numerically Koszul for every symmetric quiver $Q$ and conjecture that it is in fact Koszul; we also prove this conjecture for quivers of a certain class.
\end{abstract}

\maketitle

\section{Introduction}
Motivic Donaldson-Thomas invariants of (symmetric) quivers, abbreviated DT invariants in the following, were introduced in 
\cite{kontsevich_cohomological} as natural analogues of the Donaldson-Thomas invariants counting classes of sheaves. 
They are originally defined purely formally via Euler product factorizations of motivic generating series of stacks of 
representations of the quiver. To investigate their properties, more structural interpretations are thus desirable. 
The first such interpretation was already given in \cite{kontsevich_cohomological}, by introducing the Cohomological Hall algebra of a quiver. 
This algebra was conjectured in \cite{kontsevich_cohomological}, and proven in \cite{Efimov}, to be isomorphic to the free supercommutative algebra
generated by a certain subspace whose Poincare series encodes all DT invariants, thus proving their integrality and positivity. 
Geometric interpretations of the DT invariants in terms of moduli spaces of quiver representations were subsequently 
derived in \cite{meinhardt_reineke}, \cite{franzen_reineke}. In the special case of one-vertex quivers, a 
combinatorial formula for DT invariants was derived in \cite{MR2889742}. Much more recently, the graded dual of the 
Cohomological Hall algebra was identified in \cite{dotsenko_mozgovoy} with the canonical coalgebra on the universal envelope 
of a vertex Lie algebra; since the latter is cofree, this in particular gave an alternative proof of the main result of \cite{Efimov}. 

The motivic generating series $A_Q(x,q)$ of a symmetric quiver $Q$ can be expressed in terms of the motivic DT invariants $\mathrm{DT}_\mathbf{d}(q)$, 
where $\mathbf{d}$ is a collection of nonnegative integers, one for each vertex of $Q$. Roughly speaking $A_Q(x,q)$ is equal to the plethystic exponential of 
the generating function of $\mathrm{DT}_\mathbf{d}(q)$ divided by $(1-q)$. Commonly, the plethystic exponential is thought of as a formula that relates the Poincar\'e series of a vector space to that of
the space of symmetric tensors, see, for example, \cite{getzler_mixed}. For our purposes, a slight modification of that result will be of importance: the plethystic exponential relates the Poincar\'e 
series of a Lie superalgebra to that of its universal enveloping algebra \cite{MR1360005}. 

Let us briefly explain how Lie superalgebras enter the story.
We start with defining an associative supercommutative algebra $\mathcal{A}_Q$. That algebra is defined via explicit quadratic relations on generators $a_{i,k}$, where $i$ is a vertex of $Q$ and $k$ is a
nonnegative integer. We attach certain degrees to the generators $a_{i,k}$ and show that the Poincar\'e series of $\mathcal{A}_Q$ is related to the motivic generating function $A_Q(x,q)$ by a simple change of variables. 
In the simplest non-trivial case of the quiver $Q$ with one vertex and two loops, the algebra $\mathcal{A}_Q$ is the commutative algebra in even variables $a_k$, $k\ge 0$, subject to the relations 
 \[
\sum_{k_1+k_2=k}a_{k_1}a_{k_2}=0, \qquad k\ge 0.
 \]
According to the Koszul duality theory for associative algebras, there is a quadratic algebra $\mathcal{A}_Q^!$ associated to $\mathcal{A}_Q$. Being a Koszul dual of a \emph{supercommutative} associative algebra, the algebra $\mathcal{A}_Q^!$ is isomorphic to the universal enveloping algebra of a certain Lie superalgebra $\fg_Q$. For example, for the quiver with one vertex and two loops, one obtains the Lie superalgebra with odd generators $b_k$, $k\ge 0$, subject to the relations $[b_k,b_l]=[b_{k-1},b_{l+1}]$. Since $A_Q(x,q)$ is equal to the plethystic exponential of the generating function of the DT invariants divided by $(1-q)$, it remains to find a subspace of $\fg_Q$
whose Poincar\'e series, once multiplied by $(1-q)^{-1}$, produces the Poincar\'e series of $\fg_Q$. The following theorem summarises the results proved in Section \ref{sec:bar-and-dt}; it identifies this subspace as a Lie subalgebra of $\fg_Q$, re-proving the positivity of motivic Donaldson--Thomas invariants of $Q$. To state the theorem, we use the notation  $\mathrm{Diff}_1$ for the first Weyl algebra over $\mathbb{Q}$ with its standard generators $t$ and $\partial_t$.

\begin{ithm}
The Koszul dual Lie algebra $\mathfrak{g}_Q$ has a $\mathrm{Diff}_1$-module structure, and the action of $\mathbb{Q}[t]\subset\mathrm{Diff}_1$ on $\mathfrak{g}_Q$ is free, with the space of generators $\mathrm{Ker}(\partial_t)$. Moreover, for each $\mathbf{d}\in\mathbf{Z}_{\ge 0}^{Q_0}$, we have 
 \[
\mathrm{DT}_\mathbf{d}(q)=\sum_{n\ge 0} \dim(\mathrm{Ker}(\partial_t)_{\mathbf{d},n})q^{\frac12n-1}\in\mathbb{Z}_{\ge 0}[q^{\pm\frac12}].
 \] 
\end{ithm}

As an intermediate step in our proof, we establish the following result that is of independent interest. Recall that an associative algebra is called numerically Koszul if the Poincar\'e series of its Koszul complex is equal to~$1$; in our case where all Poincar\'e series are multigraded, this means that the Poincar\'e series of the algebra $\mathcal{A}_Q$ and of its Koszul dual algebra $\mathcal{A}_Q^!$ satisfy the relation
 \[
P(\mathcal{A},x,q)P((\mathcal{A}^!)^\vee,q^{-\frac12}x,q)=1.
 \] 

\begin{ithm}
The algebra $\mathcal{A}_Q$ is numerically Koszul for every symmetric quiver $Q$.
\end{ithm}

It is known that each Koszul algebra is numericaly Koszulness, but the converse is not true in general. However, this result is a strong evidence towards the conjecture that the algebras $\mathcal{A}_Q$ are always Koszul. At the moment, we can not prove this conjecture in full generality, but we managed to classify all quivers $Q$ for which the algebra $\mathcal{A}_Q$ has a quadratic Gr\"obner basis of relations for the natural ordering of variables $a_{i,k}$ that prioritises the value of $k$ in comparing the generators, and all quivers $Q$ for which the defining relations of our algebra form a regular sequence, thus exhausting two most obvious criteria of Koszulness.    

One of the intermediate steps of our proof uses the vertex operator algebra approach to DT invariants proposed in \cite{dotsenko_mozgovoy}, which in turn may be related to the more geometric approach of \cite{joyce_ringel,joyce_arxiv,latyntsev}. However, our results are not merely rephrasing those of \cite{dotsenko_mozgovoy}: there, the motivic DT invariants are related to \emph{negative} halves of the coefficient algebras of certain vertex Lie algebras, and in our approach, studying the \emph{positive} halves is crucial, and the two are not related in any obvious way. That surprising symmetry remains a mystery that demands further explanations. 
It is also worth exploring a different layer of Koszul duality at play: algebras $\mathcal{A}_Q$ are related to certain commutative vertex algebras, and algebras $\mathfrak{g}_Q$ are related to certain vertex Lie algebras, and the two seem to be Koszul dual as vertex algebras as well. It is interesting to remark that the corresponding commutative vertex algebras and vertex Lie algebras are free in an appropriate sense (for certain locality functions on the sets of generators \cite{MR843307,Ro}), and so the Koszul duality between the algebras $\mathcal{A}_Q$ and $\mathfrak{g}_Q$ appears to mimic the classical boson-fermion correspondence in a way that appears quite different from that recently described in the recent paper \cite{davison}. It is also worth noting that if $Q$ is obtained from a simple graph $G$ by ``doubling'' (putting an arrow $i\to j$ whenever $i$ and $j$ are connected by an edge), the algebra $\mathcal{A}_Q$ may be interpreted as the coordinate algebra of the arc space \cite{MR2354631,MR1381967} of the variety defined by quadratic monomials corresponding to edges of $G$. Those algebras are studied in the recent papers \cite{bringmann2021graph,MR3456046,jenningsshaffer2020qseries,Li,li2020jet} with lattice vertex algebras as one of the main tools. It would be desirable to examine the algebras $\mathcal{A}_Q$ for other quivers in the context of the study of relationships between arc spaces and vertex algebras \cite{ArMo}. These observations will be addressed elsewhere.

\subsection*{Structure of the paper. }
The paper is organized as follows. In Section \ref{rec}, we give the necessary recollections on DT invariants, quadratic algebras, and criteria of Koszulness. 
In Section \ref{sec:quad-alg}, we define the quadratic algebras $\mathcal{A}_Q$ and establish a relationship between their Poincar\'e series and the motivic generating functions.
In Section \ref{sec:dual}, we describe the Koszul dual Lie superalgebra $\mathfrak{g}_Q$, and relate it to the known combinatorics of DT invariants for one-vertex quivers. 
In Section \ref{sec:bar-and-dt}, we define an action of the Weyl algebra on $\mathfrak{g}_Q$, and use it to prove the main results of the paper. 
In Section \ref{koszulness}, we show that the algebra $\mathcal{A}_Q$ is always numerically Koszul, and classify the quivers for which it can be proved to be Koszul using one of the standard methods.
In the appendix, we prove auxiliary linear algebra results used in the paper. 

\subsection*{Acknowledgements. } We thank Sergey Mozgovoy for useful discussions, and especially for suggesting that Theorem \ref{th:char-symmetry} must be true. 

V.~D. was supported by the University of Strasbourg Institute for Advanced Study (USIAS) through the Fellowship USIAS-2021-061 within the French national program ``Investment for the future'' (IdEx-Unistra), and by the French national research agency project ANR-20-CE40-0016. M.~R. is supported by the DFG SFB-TRR 191 ``Symplectic structures in geometry, algebra and dynamics''. E.F. was partially supported by the grant RSF 19-11-00056. The study has been partially funded within the framework of the HSE University Basic Research Program.

\section{Recollections}\label{rec}

All vector spaces and chain complexes in this article are defined over the field of rational numbers. All chain complexes are homologically graded, with the differential of degree $-1$. 

Throughout the article, $Q$ denotes a finite quiver with the set of vertices $Q_0$ and the set of arrows~$Q_1$. 
We shall assume that $Q$ is symmetric, that is, the number of arrows from $i$ to $j$ is equal to the number of arrows from $j$ to $i$ for all $i,j\in Q_0$. 
We denote by $L$ the free abelian group $\mathbb{Z}^{Q_0}$, and denote its standard basis elements by $\alpha_i$. The \emph{Euler form} of $Q$ is the bilinear form on $L$ defined as
$$\chi({\bf d},{\bf e})=\sum_{i\in Q_0}\mathbf{d}_i\mathbf{e}_i-\sum_{(a\colon i\rightarrow j) \in Q_1}\mathbf{d}_i\mathbf{e}_j.$$
Under our assumption, the Euler form is symmetric.

\subsection{Graded vector spaces and algebras}

Most vector spaces considered in this article are of the form
 \[
V=\bigoplus_{\mathbf{d}\in L}V_\mathbf{d}, \quad V_\mathbf{d}=\bigoplus_{(\mathbf{d},n)\in L\times\mathbb{Z}}V_\mathbf{d}^n .
 \]
We consider the category $\mathrm{Vect}^{L\times\mathbb{Z}}$ of such vector spaces, with morphisms being maps of degree zero. The category $\mathrm{Vect}^{L\times\mathbb{Z}}$ is monoidal, with the monoidal structures given by the tensor product $V\otimes W$ defined by
 \[
(V\otimes W)_\mathbf{d}^n=\coprod_{(\mathbf{d'},n')+(\mathbf{d''},n'')=(\mathbf{d},n)}V_\mathbf{d'}^{n'}\otimes W_{\mathbf{d''}}^{n''}.
 \]
Moreover, we can use the Koszul sign rule to define a braiding $\sigma\colon V\otimes W\to W\otimes V$ by the formula $\sigma(v\otimes w)=(-1)^{n'n''} w\otimes v$, for $v\otimes w\in V_\mathbf{d'}^{n'}\otimes W_{\mathbf{d''}}^{n''}$; this braiding makes the category $\mathrm{Vect}^{L\times\mathbb{Z}}$ symmetric monoidal. Finally, we shall use graded duals; for an object $V$, its graded dual $V^\vee$ is defined by the formula
 $$
V^\vee=\bigoplus_{\mathbf{d}\in L} V_\mathbf{d}^\vee,\qquad
V_\mathbf{d}^\vee=\bigoplus_{n\in\mathbb{Z}}(V_{\mathbf{d}}^{-n})^*.
 $$

We shall work with various algebras in the category $\mathrm{Vect}^{L\times\mathbb{Z}}$, specifically, associative, commutative, and Lie algebras. One can either define them directly, using the tensor product to talk about the structure operations and using the braiding to implement permutations of arguments, or, alternatively, one may note that the category $\mathrm{Vect}^{L\times\mathbb{Z}}$ contains the category $\mathrm{Vect}$ as a full symmetric monoidal subcategory of objects of degree zero, and so one may talk about objects in $\mathrm{Vect}^{L\times\mathbb{Z}}$ that are algebras over the classical operads $\mathsf{Ass}$, $\mathsf{Com}$, and $\mathsf{Lie}$ in $\mathrm{Vect}$. In particular, as in the case of $\mathrm{Vect}$, the free associative algebra generated by an object $X$ is the tensor algebra $\mathsf{T}(X)=\coprod_{n\ge0} X^{\otimes n}$, and the free commutative algebra generated by an object $X$ is the symmetric algebra $S(X)=\coprod_{n\ge0} (X^{\otimes n})_{\Sigma_n}$. When using the braiding defined via the Koszul sign rule, it is not uncommon to refer to the corresponding notions of algebras as associative (commutative, Lie) \emph{superalgebras}. We hope that the reader is not too perturbed by the fact that, under our convention, these algebras are just associative (commutative, Lie) algebras, just in a different symmetric monoidal category.

\subsection{Poincar\'e series}\label{sec:PoSe}

Let $V$ be a $\mathbb{Z}$-graded vector space with finite-dimensional components.  The \emph{Poincar\'e series} $P(V,q)$ is defined by the formula
 \[
P(V,q)=\sum_{k\in\bZ}\dim (V^k)(-q^{\frac12})^{-k}.
 \]
(Division by two corresponds to the standard convention used for cohomological Hall algebras, and using the exponent $-k$ comes from working with homologically graded vector spaces.) In general, this expression is an element of the vector space of doubly infinite Laurent series $\mathbb{Q}[[q^{\pm\frac12}]]$; in our work, we shall only deal with situations where it is finite on one of the sides, so that it belongs to one of the fields of formal Laurent series $\mathbb{Q}((q^{\frac12}))=\mathbb{Q}[[q^{\frac12}]][q^{-\frac12}]$ and $\mathbb{Q}((q^{-\frac12}))=\mathbb{Q}[[q^{-\frac12}]][q^{\frac12}]$. 

Let us consider formal variables $x_i$, $i\in Q_0$, and denote, for ${\bf d}\in L$, $x^{\bf d}=\prod_{i\in Q_0} x_i^{{\bf d}_i}$. 
To an object $V$ of $\mathrm{Vect}^{L\times\mathbb{Z}}$ with finite-dimensional components, we shall associate its \emph{Poincar\'e series} $P(V,x,q)$:
 \[
P(V,x,q)=\sum_{\mathbf{d}\in L} P(V_\mathbf{d},q)x^\mathbf{d}
=\sum_{\mathbf{d}\in L}\sum_{k\in\mathbb{Z}}
\dim (V_\mathbf{d}^k)(-q^{\frac12})^{-k}x^\mathbf{d}.
 \]
An important property of the Poincar\'e series is the following invariance property. If $V$ is a chain complex (with a differential $\delta$ of degree $(0,-1)$), its Poincar\'e series is the same as the Poincar\'e series of its homology: 
 \[
P(V,x,q)=P(H(V,\delta),x,q).
 \]
For certain objects in $\mathrm{Vect}^{L\times\mathbb{Z}}$, we shall use the fact that Poincar\'e series behave well with respect to various operations, including tensor products. For that, it is important to consider a smaller category. Our choice is to consider the subcategory $\mathcal{C}$ consisting of objects 
$$
V=\bigoplus_{\mathbf{d}\in\mathbb{Z}_{\ge0}^{Q_0}}V_\mathbf{d}
=\bigoplus_{\mathbf{d}\in\mathbb{Z}_{\ge0}^{Q_0}}\bigoplus_{k\in\mathbb{Z}}V_\mathbf{d}^k$$
such that all components $V_\mathbf{d}^k$ are finite-dimensional and $V_\mathbf{d}^k=0$ for $k\gg0$.
The Poincar\'e series is a ring homomorphism 
 \[
K_0(\mathcal{C})\to R_Q=\mathbb{Q}((q^{\frac12}))[[x_i\colon i\in Q_0]].
 \]
We emphasize that not all objects we consider belong to $\mathcal{C}$, but whenever we use the compatibility of Poincar\'e series with tensor products of vector spaces, we shall restrict ourselves to objects from $\mathcal{C}$.

\subsection{Plethystic exponential}

We shall need an operation on formal power series called the plethystic exponential. We recall one of its definitions which is not the most general one but is well suited for the ring $R_Q=\mathbb{Q}((q^{\frac12}))[[x_i\colon i\in Q_0]]$. For each $n\ge 1$, we define the map 
 \[
p_n\colon R_Q\to R_Q 
 \]
by the formula 
 \[
p_n(f)=\left.f\right|_{x_i\to x_i^n, q^{\frac12}\to q^{\frac{n}2}}.
 \]
If we denote by $\mathfrak{m}_Q$ the maximal ideal of $R_Q$, the plethystic exponential is the group isomorphism $\mathrm{Exp}\colon(\mathfrak{m}_Q,+)\to(1+\mathfrak{m}_Q,\cdot)$ defined 
by 
 \[
\mathrm{Exp}(f)=\exp\left(\sum_{n\ge 1}\frac{p_n(f)}{n}\right).
 \] 
Alternatively, this isomorphism is uniquely defined by the property
$$\mathrm{Exp}(q^k x^{\mathbf{d}})=\sum_{n\ge0}q^{nk}x^{n\mathbf{d}},\qquad k\in\frac12\mathbb{Z},\,\mathbf{d}\in\mathbb{Z}_{\ge0}^{Q_0}\setminus\{0\}.$$

The plethystic exponential is used to compute the Poincar\'e series of the symmetric algebra of an object of $\mathcal{C}$. Namely, if $V$ is an object in $\mathcal{C}$ with $V_0=0$ (so that $P(V,x,q)\in\mathfrak{m}_Q$), then we have
 \[
P(S(V),x,q)=\mathrm{Exp}(P(V,x,q)).
 \] 
Using the definition of a $\lambda$-ring, one can re-state this property saying that the Poincar\'e series is a homomorphism of $\lambda$-rings from $K_0(\mathcal{C})$ to $R_Q$, see \cite{getzler_mixed,Knutson} for details.

\subsection{Donaldson--Thomas invariants of symmetric quivers}

For a quiver $Q$, its \emph{motivic generating series} $A_Q(x,q)$ is defined by the formula 
$$A_Q(x,q)=\sum_{{\bf d}\in\mathbb{Z}_{\ge 0}^{Q_0}}\frac{(-q^{\frac12})^{-\chi({\bf d},{\bf d})}x^{\bf d}}{\prod_{i\in Q_0} (q^{-1})_{\mathbf{d}_i}},$$
where we denote $(q)_n=(1-q)\cdot\dots\cdot(1-q^n)$ (we note that $A_Q(x,q^{-1})\in R_Q$). This series is the Poincar\'e series of a certain graded vector space associated to the quiver $Q$. Namely, let us denote by $R_{\bf d}(Q)$ the space of representations of $Q$ of dimension vector ${\bf d}$, and by $G_{\bf d}$ the basis change group. Then we have 
 \[
P(\mathcal{H}_Q,x,q)=A_Q(x,q),
 \]
where $\mathcal{H}_Q$ is the shifted direct sum of equivariant cohomology of representation spaces
$$\mathcal{H}_Q=\bigoplus_{\bf d} H^{*+\chi({\bf d},{\bf d})}_{G_{\bf d}}(R_{\bf d}(Q),\mathbb{Q}).$$
The \emph{refined Donaldson--Thomas invariants} ${\rm DT}_{\bf d}(q)$ of $Q$ are defined by
$$A_Q(x,q^{-1})=\mathrm{Exp}\left(\frac{1}{1-q}\sum_{\mathbf{d}\in\mathbb{Z}_{\ge 0}^{Q_0}}(-1)^{\chi(\mathbf{d},\mathbf{d})} {\rm DT}_{\bf d}(q^{-1})x^{\bf d}\right).$$
(Note that $A_Q(x,q)$ has coefficients in $\mathbb{Q}((q^{-\frac12}))$, so we pass to $A_Q(x,q^{-1})$ in order to work with power series with coefficients in $\mathbb{Q}((q^\frac12))$, for which we defined the plethystic exponential.) 
Evaluating $A_Q(x,q^{-1})$ at Lefshetz motive, one obtains the motivic Donaldson--Thomas series of $Q$ defined in~\cite{kontsevich_cohomological}. 

Kontsevich and Soibelman defined \cite{kontsevich_cohomological} the cohomological Hall algebra of $Q$ by endowing $\mathcal{H}_Q$ with a parabolic induction type associative product. Using equivariant localization, this product admits a purely algebraic description using a certain version of the Feigin--Odesskii shuffle product \cite{MR1736164}. More specifically, for any $\bd\in\mathbb{Z}_{\ge 0}^{Q_0}$, let us define
$$\Lambda_\mathbf{d}=\mathbb{Q}[z_{i,r}\colon i\in Q_0,\, 1\le r\le \mathbf{d}_i]^{\Sigma_\mathbf{d}},\qquad
\Sigma_\mathbf{d}=\prod_{i\in Q_0}\Sigma_{\mathbf{d}_i}.$$
Then $H^*_{G_{\bf d}}(\mathrm{pt})\cong \Lambda_\mathbf{d}$, and
 \[
\mathcal{H}_{Q,\mathbf{d}}=\Lambda_\mathbf{d}[-\chi(\mathbf{d},\mathbf{d})].
 \]

Efimov showed in \cite{Efimov} that, after a certain (non-canonical) sign twist of the shuffle product of $\mathcal{H}_Q$, there exists an $L\times\mathbb{Z}$-graded vector space $V$ with finite-dimensional components $V_\mathbf{d}=\bigoplus_{n\ge 0}V_{\mathbf{d},n}$ for which
 \[
\mathcal{H}_Q\simeq S(V[z])
 \]
as commutative algebras (where $z$ is placed in degree $(0,2)$). Comparing the Poincar\'e series, this implies that ${\rm DT}_{\bf d}(q)\in\mathbb{N}[q^{\pm 1/2}]$.  


\subsection{Koszul algebras}\label{sec:koszulalg}

We shall now recall the basics of the theory of Koszul duality for associative and commutative algebras; we invite the reader to consult \cite{LV,PP} for further details.
All vector spaces in this section belong to the category $\mathrm{Vect}^{L\times\mathbb{Z}}$ or even to its subcategory $\mathcal{C}$ discussed in Section \ref{sec:PoSe}. We use a standard convention for suspensions, according to which the homological degree shifts are implemented by the symbol $s$ of homological degree $1$ (that is, of $L\times\mathbb{Z}$-degree $(0,1)$). 

All associative algebras $\mathcal{A}$ we consider are \emph{quadratic}, that is, presented by generators and relations as $\mathcal{A}=T(V)/(R)$, where $R\subset V^{\otimes 2}$. For each such algebra, one may define the Koszul dual algebra 
$\mathcal{A}^!$ as follows: 
\[
\mathcal{A}^!=T((sV)^\vee)/(R^\perp),
\]
where $R^\perp$ is the annihilator of $R$ under the natural pairing between $((sV)^\vee)^{\otimes 2}$ and $V^{\otimes 2}$. (Note that our definition is slightly different from the usual one \cite{PP} in that we insist on considering the Koszul dual algebra with the shifted cohomological degree; this is important since for our algebras of interest the space of generators will already have its internal homological degree, and so there is no reason to pay extra price for shifting everything back after dualising.) It is easy to see that for a (super)commutative associative algebra $\mathcal{A}$, the Koszul dual $\mathcal{A}^!$ is isomorphic to the universal enveloping algebra of a certain Lie (super)algebra $\mathfrak{g}(\mathcal{A})$; we shall refer to that Lie algebra as the \emph{Koszul dual Lie algebra} of $\mathcal{A}$.

One can show that $\mathcal{A}^!$ is isomorphic to the subalgebra of the Yoneda algebra $\mathrm{Ext}^\bullet(\mathbb{Q},\mathbb{Q})$ generated by $\mathrm{Ext}^1$; here $\mathbb{Q}$ is considered as an $\mathcal{A}$-module on both sides via the augmentation $\mathcal{A}\to\mathbb{Q}$ annihilating all generators. An algebra $\mathcal{A}$ is said to be Koszul if the inclusion $\mathcal{A}^!\hookrightarrow\mathrm{Ext}^\bullet(\mathbb{Q},\mathbb{Q})$ is an isomorphism. It is also important for us that there exists a chain complex $\mathsf{K}(A)$ with the underlying graded vector space $\mathcal{A}\otimes(\mathcal{A}^!)^\vee$ which is a resolution of $\mathbb{Q}$ if and only if the algebra $\mathcal{A}$ is Koszul. In particular, for a Koszul algebra $\mathcal{A}$, we have 
\begin{equation}\label{eq:Froberg}
P(\mathcal{A},x,q)P((\mathcal{A}^!)^\vee,q^{-\frac12}x,q)=1
\end{equation}
(the factor $q^{-\frac12}$ accounts for the fact that we consider the generating functions that also count the homological degrees, and the Koszul dual algebra is generated by $(sV)^\vee$). An algebra satisfying this property is said to be numerically Koszul; thus a Koszul algebra is numerically Koszul, but the converse is not true in general. 

\subsection{Criteria of Koszulness}

There are two useful criteria of Koszulness which we shall now recall. 

Let us choose a basis $X$ in the space of generators $V$. An ordering of monomials (that is, of words in the alphabet $X$) in the tensor algebra $T(V)$ is said to be admissible if it is a total well ordering, and if the product is an increasing function of its arguments: replacing one of the monomials in a product by a greater one increases the result. Given an admissible ordering of monomials in the tensor algebra, one says that a subset $G$ of an ideal $I\subset T(V)$ is a (noncommutative) Gr\"obner basis if the leading monomial of every element of $I$ is divisible by a leading monomial of an element of~$G$. The primary reason to look for Gr\"obner bases is dictated by considerations of linear algebra: a Gr\"obner basis for an ideal gives extensive information on the quotient modulo $I$. More precisely, a monomial is said to be normal with respect to $G$ if it is not divisible by any of the leading monomials of elements of~$G$. The normal monomials with respect to any set of generators of an ideal $I$ always form a spanning set of the quotient modulo $I$. However, a generating set $G$ of $I$ is a Gr\"obner basis of $I$ if and only if cosets of monomials that are normal with respect to $G$ form a basis of the quotient modulo $I$, see \cite[Prop.~2.3.3.5]{BD}. Among different choices of a Gr\"obner basis, there is the so called reduced Gr\"obner basis, for which no leading term is a divisor of another one, and all non-leading terms are normal; for each admissible ordering, there exists a unique reduced Gr\"obner basis. 

\begin{prop}[{\cite[Chapter 4, Theorem~3.1]{PP}}]\label{prop:KoszulGB}
An algebra $\mathcal{A}$ with a Gr\"obner basis consisting of elements of weight~$2$ is Koszul. 
\end{prop}

Basic linear algebra implies that if $\mathcal{A}$ has a Gr\"obner basis consisting of elements of weight $2$ (for a certain admissible ordering of monomials), then the Koszul dual algebra $\mathcal{A}^!$ has a quadratic Gr\"obner basis for the ordering obtained by the opposite ordering of monomials in elements of the dual basis.

The next criterion is only applicable to commutative algebras. Recall that if $R$ is a commutative ring and $M$ is an $R$-module, an element $f$ of even homological degree is said to be $M$-regular if $f$ is injective on $M$. This definition has an odd counterpart \cite{MR739100}: an element $f$ of odd homological degree is said to be $M$-regular if the kernel $f$ on $M$ coincides with the image of $f$ on $M$ (in other words, is also trivial in the obvious sense). A sequence of elements $f_1,\ldots,f_k$ of $R$ is said to be $M$-regular if $f_i$ is regular on $M/(f_1,\ldots,f_{i-1})M$ for all $i=1,\ldots,k-1$.
For $M=R$ we will use just the word ``regular''.

\begin{prop}\label{prop:KoszulRS}
A commutative algebra $A$ which is the quotient of some polynomial algebra $R$ by an ideal $(f_1,\ldots,f_m)$ with quadratic relations $f_1,\ldots,f_m$ that form a regular sequence is Koszul.
\end{prop}

\begin{proof}
This immediately follows from the Koszul complex criterion of regularity \cite[Sec.~2.3]{MR1005697} and uniqueness of minimal resolutions of weight graded algebras.
\end{proof}

\subsection{Gr\"obner--Shirshov bases for Lie algebras}

We shall make use of the technique essentially going back to Shirshov for determining bases of Lie algebras presented by generators and relations.
Let us briefly recall the corresponding definitions; we refer the reader to \cite{MR1700511,MR1733167} for details. 

Suppose that $X$ is a set equipped with a well-order. We can consider the free monoid $\langle X\rangle$ generated by $X$, consisting of all words in the alphabet $X$ with the associative product of each two elements given by concatenation. We shall moreover assume that there is a parity function $X\to \mathbb{Z}/2\mathbb{Z}=\{0,1\}$ allowing to write $X=X_0\sqcup X_1$; we extend parity to $\langle X\rangle$ additively. A non-empty word $w$ is said to be a \emph{Lyndon--Shirshov word} if it is the strictly largest one (with respect to the graded lexicographic ordering of $\langle X\rangle$ induced by the order of $X$) among all its cyclic shifts. 
Furthermore, a non-empty word $w$ is said to be a \emph{super-Lyndon--Shirshov word} if it is a Lyndon--Shirshov word or a square of an odd Lyndon--Shirshov word. 

We can also consider the free magma $M(X)$ generated by $X$, consisting of all nonassociative words in the alphabet $X$ with the product of each two elements given by enclosing their concatenation in brackets. We extend parity to $M(X)$ additively. There is an obvious morphism 
 \[
M(X)\to \langle X\rangle,\quad w\mapsto \underline{w},
 \]
which erases all brackets. A non-empty nonassociative word $w$ is said to be a \emph{Lyndon--Shirshov monomial} if either $w\in X$ or 
\begin{itemize}
\item if $w=(w_1w_2)$ then $w_1$ and $w_2$ are Lyndon--Shirshov monomials and $\underline{w_1}$ is lexicographically greater than $\underline{w_2}$,
\item if $w=((w_1w_2)w_3)$ then $\underline{w_3}$ is lexicographically greater than $\underline{w_2}$ or equal to it. 
\end{itemize}
A non-empty nonassociative word $w$ is said to be a \emph{super-Lyndon--Shirshov monomial} if it is a Lyndon--Shirshov monomial or a square of an odd Lyndon--Shirshov monomial.

It is known that for every super-Lyndon--Shirshov monomial $w$, the associative word $\underline{w}$ is a super-Lyndon--Shirshov word, and that for any super-Lyndon--Shirshov word there is a unique bracketing making it a super-Lyndon--Shirshov monomial. Moreover, let $\mathfrak{g}$ be a Lie algebra with generators $X$ and some relations $R$, and let us interpret elements of $R$ as linear combinations of commutators in the free algebra generated by $X$, giving a presentation of the universal envelope $U(\mathfrak{g})$. The Lie algebra $\mathfrak{g}$ has a basis of cosets of super-Lyndon--Shirshov monomials $w$ for which $\underline{w}$ is normal with respect to $R$ if and only if $R$ is a Gr\"obner basis of defining relations of $U(\mathfrak{g})$. 

\section{Quadratic algebras associated to symmetric quivers}\label{sec:quad-alg}

Let $Q$ be a symmetric quiver. The protagonist of the paper is a certain associative algebra associated to $Q$. We denote by $m_{i,j}$ the number of arrows from $i$ to $j$ in $Q$. 

\begin{definition}
The algebra $\mathcal{A}_Q$ is defined as follows. Its space of generators $V_Q$ has a basis $a_{i,k}$ with $i\in Q_0$, $k\ge 0$; we set $\deg(a_{i,k})=(\alpha_i,-2k-m_{i,i})\in L\times\mathbb{Z}$, so $V_Q$ is an object of $\mathcal{C}$. There are two groups of defining relations of $\mathcal{A}_Q$:
\begin{gather*}
a_{i,k_1}a_{j,k_2}=(-1)^{m_{i,i}m_{j,j}} a_{j,k_2}a_{i,k_1} \quad \text{ for all } i,j,k_1,k_2,\\
\sum_{k_1+k_2=k} \binom{k_2}{p} a_{i,k_1} a_{j,k_2} =0\quad \text{ for all } k\ge 0 \text{ and } 0\le p<m_{i,j}.  
\end{gather*}
\end{definition}

Note that all relations of the algebra $\mathcal{A}_Q$ are manifestly homogeneous, and therefore that algebra is itself an object of $\mathcal{C}$. The first group of relations merely state that $\mathcal{A}_Q$ is (super--)commutative. Relations of the second group should be thought of as follows. Consider the formal generating series 
\[
a_i(z)=\sum_{k\ge 0} a_{i,k}z^k.
\] 
Then the relations
 \[
\sum_{k_1+k_2=k} \binom{k_2}{p} a_{i,k_1} a_{j,k_2} =0\quad \text{ for all } k\ge 0   
 \]
are, up to the factor $\frac{1}{p!}$, the coefficients of the power series  
 \[
a_i(z)\frac{d^p\phantom{z}}{dz^p}a_j(z)=0 . 
 \]
It is also easy to show that this group of relations is equivalent to a bigger more symmetric one 
 \[
\frac{d^p\phantom{z}}{dz^p}a_i(z)\frac{d^q\phantom{z}}{dz^q}a_j(z)=0 \text{ for } 0\le p+q<m_{ij}.
 \] 

\begin{rem}
In a particular case where for all $i,j\in Q_0$ there is at most one arrow from $i$ to $j$ (to such quiver we may associate in an obvious way a graph $\Gamma$ without multiple edges), the algebra $\mathcal{A}_Q$ was previously studied in the literature as the coordinate algebra defining the arc scheme of the graph scheme of $\Gamma$, see, for example, \cite{bringmann2021graph,MR3456046,jenningsshaffer2020qseries,Li,li2020jet}. 
\end{rem}

Let us give several examples of algebras $\mathcal{A}_Q$. 

\begin{example}\leavevmode
\begin{enumerate} 
\item Let $Q$ be the $A_1$ quiver. Then the algebra $\mathcal{A}_Q$ is generated by $a_0,a_1,\dots$; the only relations are those of the first group, and our algebra is the polynomial algebra $\mathbb{Q}[a_0,a_1,\dots]$.
\item Let $Q$ be the quiver with one vertex and one loop. Then the algebra $\mathcal{A}_Q$ is generated by $a_0,a_1,\dots$. Relations of the first group say that the generators anti-commute, and the relation $a(z)a(z)=0$ of the second group is redundant, since it follows from anti-commutativity. Thus, our algebra is the infinite Grassmann algebra $\Lambda(a_0,a_1,\dots)$. 
\item Let $Q$ be the quiver with one vertex and two loops. Then the algebra $\mathcal{A}_Q$ is generated by commuting generators $a_0, a_1, \dots$ modulo the relations 
\begin{gather*}
(a_0+a_1z+a_2z^2+\dots)^2=0,\\
(a_0+a_1z+a_2z^2+\dots)\frac{d\phantom{z}}{dz}(a_0+a_1z+a_2z^2+\dots)=0.
\end{gather*}
The second relation follows from the first one by differentiation, so our algebra is 
 \[
\frac{\mathbb{Q}[a_0,a_1,\dots]}{\left(a_0^2, 2a_0a_1, 2a_0a_2+a_1^2, \ldots, \sum\limits_{i+j=k} a_ia_j\right)}.
 \]
This algebra was studied in many papers in the context of level $1$ modules over the Lie algebra $\widehat{\mathfrak{sl}_2}$ (see e.g. \cite{FS}). 
\item Let $Q$ be the quiver with two vertices and one arrow in each of the two directions between these vertices. Then the algebra $\mathcal{A}_Q$ is generated by $a_{0,0}, a_{0,1},\dots$ and
$a_{1,0}, a_{1,1},\dots$. Relations of the first group say that the generators commute, and relations of the second group say that $a_0(z)a_1(z)=0$. This algebra was studied in \cite{MR3973134}, where an unexpected algebraic property of its nilradical was unravelled.
\end{enumerate}
\end{example}

We shall now describe a convenient interpretation of the graded dual space $\mathcal{A}_Q^\vee$ of the algebra $\mathcal{A}_Q$, inspired by the approach of \cite{FS} in the case of the algebra associated to the quiver with one vertex and two loops. (For the case of jet algebras of graph schemes, this result is established in \cite[Sec.~5.3]{Li}.) Recall that we denote $\Lambda_\mathbf{d}=\mathbb{Q}[z_{i,r}\colon i\in Q_0,\, 1\le r\le \mathbf{d}_i]^{\Sigma_\mathbf{d}}$.

\begin{prop}\label{prop:functional-realisation}
 Let $F_\mathbf{d}\in \mathbb{Q}[z_{i,r}\colon i\in Q_0,\, 1\le r\le \mathbf{d}_i]$ be defined by the formula
 \[
F_\mathbf{d}:=\prod_{i\in Q_0}\prod_{1\le r<r'\le \mathbf{d}_i}(z_{i,r}-z_{i,r'})^{m_{i,i}} \prod_{\substack{\{i,i'\}\in Q_0, i\ne i'\\ 1\le r\le \mathbf{d}_i,\\1\le r'\le \mathbf{d}_{i'}}}(z_{i,r}-z_{i',r'})^{m_{i,i'}}.
 \]
Then we have an isomorphism of graded vector spaces 
 \[
\mathcal{A}_{Q,\mathbf{d}}^\vee\cong F_\mathbf{d}\Lambda_{\mathbf{d}}[-\mathbf{m}\cdot\mathbf{d}], 
 \]
where $\mathbf{m}\cdot\mathbf{d}=\sum_{i\in Q_0} m_{i,i}\mathbf{d}_i$. 
\end{prop}

\begin{proof}
Note that the element $a_{i,k}z_{i,s}^k$ is of cohomological degree $-m_{i,i}$, so the expression 
 \[
\prod_{i\in Q_0} a_i(z_{i,1})\dots a_i(z_{i,\mathbf{d}_i})
 \]
is of cohomological degree $-\mathbf{m}\cdot\mathbf{d}$, and thus for each element 
 \[
\xi\in \mathcal{A}_{Q,\mathbf{d}}^\vee ,
 \]
the evaluation 
 \[
f_\xi=\xi \Bigl(\prod_{i\in Q_0} a_i(z_{i,1})\dots a_i(z_{i,\mathbf{d}_i})\Bigr)
 \]
is a map of graded vector spaces from $\mathcal{A}_{Q,\mathbf{d}}^\vee$ to the degree shifted polynomial ring 
 \[
\mathbb{Q}[z_{i,r}\colon  i\in Q_0,\, 1\le r\le \mathbf{d}_i][-\mathbf{m}\cdot\mathbf{d}]
 \] 
(we get a polynomial and not a power series since $\xi$ belongs to the graded dual). \\

Let us show that  $f_\xi\in F_\mathbf{d}\Lambda_{\mathbf{d}}[-\mathbf{m}\cdot\mathbf{d}]$. We note that the second group of relations of the algebra $\mathcal{A}_Q$ imposes vanishing conditions on the diagonals where pairs of variables $z_{i,r}$ coincide. First, for each $i\in Q_0$, the polynomial $f_\xi$ is manifestly divisible by $(z_{i,r}-z_{i,r'})^{m_{i,i}}$ for all $1\le r<r'\le \mathbf{d}_i$. Moreover, commutativity of generators ensures that $f_\xi$ is symmetric in variables $z_{i,1}$,\ldots, $z_{i,\mathbf{d}_i}$ for even $m_{i,i}$ and anti-symmetric for odd $m_{i,i}$. Thus, the ratio $f_\xi/\prod_{1\le r<r'\le\mathbf{d}_i}(z_{i,r}-z_{i,r'})^{m_{i,i}}$ is $\Sigma_{\mathbf{d}_i}$-invariant. Next, for $i\ne j\in Q_0$, the polynomial $f_\xi$ is divisible by $(z_{i,r}-z_{i',r'})^{m_{i,i'}}$ for $1\le r\le \mathbf{d}_i$, $1\le r'\le \mathbf{d}_{i'}$. Overall, the polynomial $F_\mathbf{d}$ divides $f_\xi$ and the ratio is $\Sigma_\mathbf{d}$-invariant, as required. To show that every polynomial of this form can be obtained as $f_\xi$ for a suitable $\xi\in \mathcal{A}_Q({\bf d})^*$, one just reverses the logic of the proof we just presented. 
\end{proof}

\begin{cor}\label{cor:series-dual}
We have
 \[
P(\mathcal{A}_Q,x,q)=\sum_{\mathbf{d}\in\mathbb{Z}_{\ge 0}^{Q_0}} \frac{(-q^{\frac12})^{\mathbf{d}\cdot\mathbf{d}-\chi(\mathbf{d},\mathbf{d})}}{\prod_{i\in Q_0} (q)_{\mathbf{d}_i}} x^{\bf d}.
 \]
\end{cor}

\begin{proof}
According to Proposition \ref{prop:functional-realisation}, we have
 \[
\mathcal{A}_{Q,\mathbf{d}}^\vee\cong F_\mathbf{d}\Lambda_{\mathbf{d}}[-\mathbf{m}\cdot\mathbf{d}]. 
 \] 
Clearly, $P(\Lambda_{\mathbf{d}},q)=\frac1{\prod_{i\in Q_0} (q^{-1})_{\mathbf{d}_i}}$. Each variable $z_{i,s}$ is of degree $2$, so the degree of the polynomial $F_\mathbf{d}$ is equal to
 \[
2\sum_{i\in Q_0}m_{i,i}\binom{\mathbf{d}_i}2 +\sum_{i\ne j \in Q_0}m_{i,j}\mathbf{d}_i \mathbf{d}_j=\sum_{i,j\in Q_0}m_{i,j}\mathbf{d}_i \mathbf{d}_j-\mathbf{m}\cdot\mathbf{d},
 \]
and taking the account the extra shift $-\mathbf{m}\cdot\mathbf{d}$, we find that  
 \[
P(\mathcal{A}_Q^\vee,x,q)=\sum_{\mathbf{d}\in\mathbb{Z}_{\ge 0}^{Q_0}} \frac{(-q^{\frac12})^{-\sum_{i,j\in Q_0}m_{i,j}\mathbf{d}_i \mathbf{d}_j} }{\prod_{i\in Q_0} (q^{-1})_{\mathbf{d}_i}}x^{\bf d}=\sum_{\mathbf{d}\in\mathbb{Z}_{\ge 0}^{Q_0}} \frac{(-q^{\frac12})^{\chi(\mathbf{d},\mathbf{d})-\mathbf{d}\cdot\mathbf{d}}}{\prod_{i\in Q_0} (q^{-1})_{\mathbf{d}_i}} x^{\bf d}. 
 \]
Finally, 
 \[
P(\mathcal{A}_Q,x,q)=P(\mathcal{A}_Q^\vee,x,q^{-1})=
\sum_{\mathbf{d}\in\mathbb{Z}_{\ge 0}^{Q_0}} \frac{(-q^{\frac12})^{\mathbf{d}\cdot\mathbf{d}-\chi(\mathbf{d},\mathbf{d})}}{\prod_{i\in Q_0} (q)_{\mathbf{d}_i}} x^{\bf d},
 \]
as required. 
\end{proof}

We are now ready to unravel the true reason why the algebra $\mathcal{A}_Q$ is of interest to us. For that, we note that though the motivic generating series 
 $$
\sum_{{\bf d}\in\mathbb{Z}_{\ge 0}^{Q_0}}\frac{(-q^{\frac12})^{-\chi({\bf d},{\bf d})}x^{\bf d}}{\prod_{i\in Q_0} (q^{-1})_{\mathbf{d}_i}}
 $$
is an element of the ring $\mathbb{Q}((q^{-\frac12}))[[x_i\colon i\in Q_0]]$, and the Poincar\'e series of the algebra $\mathcal{A}_Q$ is an element of the ring $R_Q=\mathbb{Q}((q^{\frac12}))[[x_i\colon i\in Q_0]]$, each of those series can be viewed as an element of the ring $\mathbb{Q}(q^{\frac12})[[x_i\colon i\in Q_0]]$. Note that if we view the two former rings as subspaces of $\mathbb{Q}[[q^{\pm\frac12}]][[x_i\colon i\in Q_0]]$, the embeddings of $\mathbb{Q}(q^{\frac12})[[x_i\colon i\in Q_0]]$ into them gives two different embeddings 
 \[
\mathbb{Q}(q^{\frac12})[[x_i\colon i\in Q_0]]\hookrightarrow\mathbb{Q}[[q^{\pm\frac12}]][[x_i\colon i\in Q_0]] .
 \]

\begin{prop}\label{prop:poincare series}
In the ring $\mathbb{Q}(q^{\frac12})[[x_i\colon i\in Q_0]]$, we have 
 \[
A_Q(x,q)=P(\mathcal{A}_Q,q^{\frac12}x,q).
 \]
\end{prop}

\begin{proof}
According to Corollary \ref{cor:series-dual}, we have 
 \[
P(\mathcal{A}_Q,x,q)=\sum_{\mathbf{d}\in\mathbb{Z}_{\ge 0}^{Q_0}} \frac{(-q^{\frac12})^{\mathbf{d}\cdot\mathbf{d}-\chi(\mathbf{d},\mathbf{d})}}{\prod_{i\in Q_0} (q)_{\mathbf{d}_i}} x^{\bf d}.
 \]
At the same time, we have the following equalities in $\mathbb{Q}(q^{\frac12})[[x_i\colon i\in Q_0]]$:
\begin{multline*}
A_Q(x,q)=\sum_{{\bf d}\in\mathbb{Z}_{\ge 0}^{Q_0}}\frac{(-q^{\frac12})^{-\chi({\bf d},{\bf d})}x^{\bf d}}{\prod_{i\in Q_0} 
(q^{-1})_{\mathbf{d}_i}}=
 \sum_{{\bf d}\in\mathbb{Z}_{\ge 0}^{Q_0}}\frac{(-q^{\frac12})^{-\chi({\bf d},{\bf d})}q^{\sum_{i\in Q_0}\binom{\mathbf{d}_i+1}2}x^{\bf d}}{\prod_{i\in Q_0}\left( (q-1)\cdot\ldots\cdot(q^{\mathbf{d}_i}-1)\right)}=\\
 \sum_{{\bf d}\in\mathbb{Z}_{\ge 0}^{Q_0}}\frac{(-1)^{-\chi({\bf d},{\bf d})+\sum_{i\in Q_0}\mathbf{d}_i}  q^{-\frac12\chi({\bf d},{\bf d})+\sum_{i\in Q_0}\binom{\mathbf{d}_i+1}2}x^{\bf d}}{
\prod_{i\in Q_0} (q)_{\mathbf{d}_i}}.
\end{multline*}

We note that 
 \[
-\chi({\bf d},{\bf d})+\sum_{i\in Q_0}\mathbf{d}_i\equiv -\chi({\bf d},{\bf d})+\sum_{i\in Q_0}\mathbf{d}_i^2 =\mathbf{d}\cdot\mathbf{d}-\chi({\bf d},{\bf d})\pmod{2}
 \]
and that 
 \[
-\frac12\chi({\bf d},{\bf d})+\sum_{i\in Q_0}\binom{\mathbf{d}_i+1}2=
-\frac12\chi({\bf d},{\bf d})+\frac12\mathbf{d}\cdot\mathbf{d}+\frac12\sum_{i\in Q_0}\mathbf{d}_i.
 \]
We see that the only difference between $A_Q(x,q)$ and $P(\mathcal{A}_Q,q^{\frac12}x,q)$ is that for each $\mathbf{d}$, the corresponding term in $A_Q(x,q)$ has the extra factor $q^{\frac12\sum_{i\in Q_0}\mathbf{d}_i}$. Since the multiplication of the monomial $x^\mathbf{d}$ by $q^{|\mathbf{d}|}=q^{\frac12\sum_{i\in Q_0}\mathbf{d}_i}$ can be implemented by the rescaling $x\mapsto q^{\frac12}x$, we conclude that
 \[
A_Q(x,q)=P(\mathcal{A}_Q,q^{\frac12}x,q),
 \]
as required. 
\end{proof}

\section{Koszul dual Lie algebras and their properties}\label{sec:dual}

Proposition \ref{prop:poincare series} explains why it is natural to seek for an interpretation of the refined Donaldson--Thomas invariants in terms of the algebra $\mathcal{A}_Q$. In this section, we make the crucial step towards the main result of the paper, and perform an extensive study of the Koszul dual Lie algebra. Note that all our algebras are defined within the same symmetric monoidal category, and the Koszul sign rule applies whenever one talks about (skew-)symmetry. 

\subsection{The description of the Koszul dual Lie algebra}

\begin{prop}\label{Koszuldual}
The Koszul dual Lie algebra $\mathfrak{g}(\mathcal{A}_Q)$ is isomorphic to the Lie algebra $\fg_Q$ generated by elements $b_{i,k}$ of degree $(\alpha_i,2k+m_{i,i}+1)$, $i\in Q_0$, $k\ge 0$ subject to the relations 
 \[
[b_{i,k},b_{i,l}]=0
 \]
for all $i\in Q_0$ with $m_{i,i}=0$ and all $k,l\ge 0$,
\[
\sum_{p=0}^{m_{i,i}-1} (-1)^p\binom{m_{i,i}-1}{p} [b_{i,k-p},b_{j,l+p}]=0, 
\]
for all $i\in Q_0$ with $m_{i,i}\ge 1$ and all $k\ge m_{i,i}-1$, $l\ge 0$, and
\[
\sum_{p=0}^{m_{i,j}} (-1)^p\binom{m_{i,j}}{p} [b_{i,k-p},b_{j,l+p}]=0 
\]
for all $i\ne j\in Q_0$ and all $k\ge m_{i,j}$, $l\ge 0$.
\end{prop}

\begin{proof}
A direct calculation shows that, if we interpret $b_{i,k}$ as $a_{i,k}^\vee$, the relations of $\mathfrak{g}_Q$ in which each Lie bracket is written as commutator, annihilate the relations of $\mathcal{A}_Q$ under the usual pairing. It is thus enough to establish that they are linearly independent and span, for each $\mathbb{Z}_{\ge 0}^{Q_0}\times \frac12\mathbb{Z}$-homogeneous component, the space of the same dimension as the space of quadratic elements of $\mathcal{A}_Q$.  

Let us first consider the homogeneous component of degree $(\alpha_i+\alpha_j,d)$ with $i\ne j$. Our relations are indexed by $k\ge m_{i,j}$, $l\ge 0$, and we have $d=k+l$, so there are relations for $d\ge m_{i,j}$, and the number of relations is equal to $d-m_{i,j}+1$. On the other hand, according to Proposition \ref{prop:functional-realisation}, the dual space of $\mathcal{A}_{Q, \alpha_i+\alpha_j}$ is identified with the space of polynomials in two variables $z_1,z_2$ divisible by $(z_1-z_2)^{m_{i,j}}$, so we conclude that the number of relations is equal to the dimension of the corresponding space of quadratic elements. These relations are linearly independent by direct inspection. 

The case of the homogeneous component of degree $(2\alpha_i,d)$ is similar, but the dual space is the space of symmetric polynomials in two variables $z_1,z_2$ divisible by $(z_1-z_2)^{m_{i,i}}$, and there are duplicates among the relations (for each $d\ge m_{i,i}$, the $k$-th relation and the $(m_{i,i}-k)$-th relation are proportional), and these two phenomena ``compensate'' each other, leading to the matching dimensions.
\end{proof}

Let us give a complete description of the Lie superalgebra $\fg_Q$ in one particular case. Let $Q$ be the quiver consisting of two vertices with no loops and one edge in each direction. The Lie superalgebra $\fg_Q$ is generated by odd elements $b_{i,k}$, $i=1,2$,  $k\ge 0$ subject to the relations
\begin{gather*}
[b_{1,k},b_{1,k}]=0,\\
[b_{1,k},b_{2,l}]=[b_{1,k-1},b_{2,l+1}],\\
[b_{2,k},b_{2,l}]=0,\\
\end{gather*}
The second group of relations means that the element $[b_{1,k},b_{2,l}]$ only depends on the sum $k+l$, and we can denote it $c_{k+l}$. Moreover, it follows from the Jacobi identity that $[b_{1,k},c_l]=[b_{2,k},c_l]=[c_k,c_l]=0$, and therefore the Lie algebra $\fg_Q$ has a basis consisting of $b_{1,k}$ (of degree $(\alpha_1,2k+1)$), $b_{2,k}$ (of degree $(\alpha_2,2k+1)$), and $c_k$ (of degree $(\alpha_1+\alpha_2,2k+2)$). Thus, the Poincar\'e 
 series of $\fg_Q$ in this case is equal to
\[
(1-q^{-1})^{-1} (-q^{-1/2}z_1 - q^{-1/2}z_2 + q^{-1} z_1z_2),
\] 
exhibiting the same shape as the plethystic logarithm of the motivic generating function. We shall now see that this is not a coincidence.

\subsection{The case of one-vertex quivers}

In this section, we recall the combinatorial approach of \cite{MR2889742} in the case of one-vertex quivers and highlight its relationship to our work. 

Let $Q$ be the quiver with one vertex and $m\ge 1$ loops. The Lie superalgebra $\fg_Q$ is generated by elements $b_k$, $k\ge 0$, of parity congruent to $m-1$ modulo $2$, subject to the relations 
 \[
\sum_{i=0}^{m-1}(-1)^i\binom{m-1}{i}[b_{k-i},b_{l+i}]=0, \quad k\ge m-1, l\ge 0.
 \]
Lattice vertex operator realisations of these algebras were studied in \cite{Do}, where it was essentially established that for the degree-lexicographic ordering arising from the ordering $b_{i}<b_{j}$ for $i>j$ of generators, the defining relations form a Gr\"obner basis of the universal enveloping algebra $U(\mathfrak{g}_Q)$. It follows that $\mathfrak{g}_Q$ has a basis parametrized by the set $\mathcal{S}$ of super-Lyndon--Shirshov words $b_{p_1}\cdots b_{p_n}$ such that $p_{i+1}\leq p_i+m-1$ for all $i<n$. 

Let us recall the partition combinatorics of \cite[Section 5]{MR2889742}. Let $\Lambda$ be the set of all partitions $\lambda$ with non-negative parts, written $$\lambda=(0\leq\lambda_1\leq\ldots\leq\lambda_n),$$ and set $l(\lambda)=n$. We define the shift operators $S^p$ on $\Lambda$ by $$S^p\lambda=(\lambda_1+p,\ldots,\lambda_n+p).$$

For two partitions $\lambda,\mu\in\Lambda$, we define $\lambda\cup\mu$ as the partition with parts $\lambda_1,\ldots,\lambda_{l(\lambda)},\mu_1,\ldots,\mu_{l(\mu)}$, listed in the non-decreasing order. We shall consider the monoid structure on $\Lambda$ defined by the formula
$$\lambda*\mu=\lambda\cup S^{(m-1)l(\lambda)}\mu.$$
Let us consider the subset $T\subset\Lambda$ consisting of all partitions $\lambda$ such that $\lambda_i\leq(m-1)(i-1)$ for all $i=1,\ldots,l(\lambda)$. It is a submonoid of $(\Lambda,*)$, and when multiplying two partitions in $T$, the monoidal product is given by the concatenation (no resorting of the parts has to be performed). Furthermore, let us consider the subset $T^0\subset T$ consisting of all partitions  $\lambda$ such that $\lambda_i<(m-1)(i-1)$ for all $i=2,\ldots,l(\lambda)$. It is totally ordered by the standard lexicographic ordering; we define the parity of $\lambda\in T^0$ to be congruent to $(m-1)l(\lambda)$ modulo $2$.
Finally, we define $T^{L,+}$ as the set of all super-Lyndon--Shirshov words \emph{in the alphabet $T^0$} (this uses the parity we just defined). 

\begin{prop} 
The set $\mathcal{S}$ is in bijection with $\mathbb{Z}_{\geq 0}\times T^{L,+}$ by associating to $(p,\lambda)\in \mathbb{Z}_{\geq 0}\times T^{L,+}$ the word
$$\left(p+ (m-1)(i-1)-\lambda_i\right)_i.$$
\end{prop}

\begin{proof}
Let us first consider the set of all monomials in variables $b_i$. This set has a subset $C$ of monomials $b_{p_1}\cdots b_{p_n}$ such that the following two conditions hold:
\begin{enumerate}
\item $p_1=0$,
\item $p_{i+1}\leq p_i+m-1$ for all $i<n$. 
\end{enumerate}

It is easy to see that the map $\imath\colon T\to C$ defined by associating to a partition $\lambda\in T$ the word $b_{p_1}\cdots b_{p_n}$ with $p_i=(m-1)(i-1)-\lambda_i$ is bijective, and in fact defines an isomorphism of monoids, that is, $\imath(\lambda*\mu)$ is the concatenation of $\imath(\lambda)$ and $\imath(\mu)$, for $\lambda,\mu\in T$. Under this bijection, the subset $T^0$ maps to the set $C^0$ of all words $a$ in $C$ such that $a_i>0$ for all $i>1$. Now let $v=b_{p_1}\cdots b_{p_n}$ be any word in $\mathcal{S}$. Since $v$ is super-Lyndon--Shirshov, $p_1$ is minimal among all $p_i$, so we can subtract $p_1$ from all $p_i$ to obtain a word $w\in \mathcal{S}$ starting from $b_0$. Such a word admits a canonical factorization $w=w^1\cdots w^k$ with all $w^i\in C_0$: the first letters of the words $w^i$ correspond precisely to occurrences of $b_0$ in $w$. Thus $w$ can be viewed as a word in the alphabet $C_0$. 

An observation that is crucial for our purposes is that since $w$ is a super-Lyndon--Shirshov word in the alphabet $\{b_i\}$, it is also a super Lyndon--Shirshov word in the alphabet $C_0$. Conversely, a super Lyndo--Shirshov word in the alphabet $C_0$ is also super-Lyndon--Shirshov as a word in the alphabet $\{b_i\}$ (this holds true since the only cyclic rotations of $w$ for which the Lyndon--Shirshov property is non-trivial are those which also start with $0$). Since the map $\imath\colon T\to C$ is a monoid isomorphism inducing a bijection between $T^0$ and $C^0$, we thus find a unique partition $\lambda\in T^{L,+}$ such that $\imath(\lambda)=w$. Adding $p=p_1$ to all its parts recovers the word $v$. The above observation also shows that, conversely, any word in $\imath(T^{L,+})$ is Lyndon. 
\end{proof}

In \cite{MR2889742}, the above argument was paired with a number-theoretic argument involving cyclotomic polynomials to establish integrality of refined DT invariants. We shall see below that there is a way to adapt the argument relying on the combinatorics of Lyndon--Shirshov words for the general case. Under the assumption of the quiver having at least one loop at every vertex, one can also generalize the partition combinatorics of \cite{MR2889742}, leading to another strategy for studying the refined DT invariants, however, the relation of this combinatorics to super-Lyndon--Shirshov words that form a basis of $\mathfrak{g}_Q$ is mysterious. We therefore postpone a description of this combinatorics to future work.

\section{DT invariants and the Lie algebra \texorpdfstring{$\mathfrak{g}_Q$}{gQ}}\label{sec:bar-and-dt}
In this section, we shall relate the Poincar\'e series of the Lie algebra $\mathfrak{g}_Q$ to the refined Donaldson--Thomas invariants of $Q$, and use that property to re-prove the positivity of those invariants. 

\subsection{The Weyl algebra action and the Poincar\'e series}

We shall start with establishing that the dimensions of the graded components of the Lie algebra $\mathfrak{g}_Q$ satisfy the inequalities expected from the refined Donaldson--Thomas invariants. For that, it would be convenient to temporarily switch from Poincar\'e series to characters; for a $\mathbb{Z}_{\ge 0}$-graded vector space, we shall denote by $\mathrm{ch}$ the generating function similar to the Poincar\'e series that ignores signs arising from the parity:
 \[
\mathrm{ch}(V,q)=\sum_{k\in\mathbb{Z}}\dim (V^k)(-q^{\frac12})^{-k}.
 \]
Our first key ingredient is the action of the first Weyl algebra $\mathrm{Diff}_1$ (generated by elements $t$ and $\partial_t$ subject to the standard relation $\partial_t t - t\partial_t=1$) on the Lie algebra $\mathfrak{g}_Q$ defined as follows. We first consider two derivations $p$ and $q$ of the free Lie algebra with generators $b_{i,k}$ which act on generators according to the formulas
 \[
p(b_{i,k})=
\begin{cases}
b_{i,k-1}, k>0,\\
\quad 0,\quad \text{otherwise},
\end{cases}
\quad
q(b_{i,k})=(k+1)b_{i,k+1}.
 \]
We note that these derivations preserve the relations of the Lie algebra $\mathfrak{g}_Q$. Indeed, the action of these derivations on the formal generating series $b_i(z):=\sum b_{i,k}z^k$ can be succinctly written as 
 \[
p(b_i(z))=zb_i(z), \quad q(b_i(z))=\partial_z b_i(z).
 \]
Rewriting the defining relations of $\mathfrak{g}_Q$ in terms of the above generating series, we note that all of them are of the form $(z-w)^{a_{ij}}[b_i(z),b_j(w)]$ with certain $a_{ij}\ge 0$, and the fact that they are preserved by our derivations is obvious, once one notes that $(\partial_z+\partial_w)(z-w)=0$. 

Since on the space of generators we have $[q,p]=\mathrm{id}$, and the commutator of two derivations is a derivation, on each component $(\mathfrak{g}_Q)_{\mathbf{d}}$ the commutator $[q,p]$ acts by multiplication by $|\mathbf{d}|$. We have $(\mathfrak{g}_Q)_{\mathbf{d}}=0$ for $\mathbf{d}=0$, and so we may divide $p$ by $|\mathbf{d}|$ and obtain the action of the first Weyl algebra $\mathrm{Diff}_1$ on each given component $(\mathfrak{g}_Q)_{\mathbf{d}}$.

This Weyl algebra action can be used to prove the following result.

\begin{prop}\label{prop:Weyl}
For each $\mathbf{d}\in\mathbb{Z}_{\ge 0}^{Q_0}$, we have 
 \[
(1-q)\mathrm{ch}((\mathfrak{g}_Q^\vee)_{\mathbf{d}},q)\in\mathbb{Z}_{\ge 0}[[q^{\frac12}]].
 \]
\end{prop}

Note that since the homological degree of $b_{i,k}$ is equal to $2k+m_{i,i}+1$, the endomorphism $t\in \mathrm{Diff}_1$ is of degree $2$, and the endomorphism $\partial_t\in \mathrm{Diff}_1$ is of degree $-2$, so $(\mathfrak{g}_Q)_{\mathbf{d}}$ is a graded module over the Weyl algebra (and is clearly concentrated in positive degrees). Let us prove a simple but useful result about graded modules over the Weyl algebra. 

\begin{lem}\label{lm:Weyl}
Each $\mathbb{Z}$-graded $\mathrm{Diff}_1$-module $M$ such that the endomorphism $t\in \mathrm{Diff}_1$ is of degree $2$, and the endomorphism $\partial_t\in \mathrm{Diff}_1$ is of degree $-2$ and such that $M_k=0$ for $k\ll 0$ is isomorphic to the free $\mathbb{Q}[t]$-module generated by $\mathrm{Ker}(\partial_t)$.
\end{lem}

\begin{proof}
Let us first show that the endomorphism $t$ is injective on $M$. Suppose that for some $v$ we have $t(v)=0$. Note that the commutation relation $[\partial_t,t]=1$ implies that $\mathbb{Q}[\partial_t](v)$ is a $\mathrm{Diff}_1$ submodule of $M$. By the boundedness of $M$ from the below, this submodule is finite-dimensional. However, the algebra $\mathrm{Diff}_1$ has no non-zero finite-dimensional modules, so $v=0$. 

Next, we shall show, by induction on $n$, that in degrees at most $n$, $M$ is generated by $\mathrm{Ker}(\partial_t)$. This is true when $n$ is the smallest integer such that $M_n\ne 0$ (all elements of that degree are automatically in the kernel of $t$). To prove the step of induction, let us note that in $\mathrm{Diff}_1$, we have, for each $m\ge 0$,
 \[
t^m\partial_t^m =t\partial_t(t\partial_t-1)\cdots (t\partial_t-m+1). 
 \]
(This well known formula is easily proved by induction on $m$.) Note that $t\partial_t$ is a degree-preserving endomorphism of $M$, so it defines an endomorphism of $M_n$. By boundedness of $M$ from the above, there exists $m>0$ such that $\partial^m$ is zero on $M_n$, which shows that as an endomorphism of $M_n$, $\partial_tt$ is annihilated by the polynomial $X(X-1)\cdots (X-m+1)$. This immediately implies that $M_n=\mathrm{Ker}(t\partial_t)\oplus\mathrm{Im}(t\partial_t)$. Moreover, $\mathrm{Im}(t\partial_t)\subset \mathrm{Im}(\partial_t)$, and $\mathrm{Ker}(t\partial_t)=\mathrm{Ker}(\partial_t)$ because of injectivity of $t$. Thus, modulo the kernel of $t$, each element of $M_n$ is obtained by action of $t$ on elements of $M_{n-1}$, ensuring that the induction step can proceed. 

To conclude, we need to show that the obvious map $\mathbb{Q}[t]\otimes\mathrm{Ker}(\partial_t)\to M$ has no kernel. If we assume the contrary and consider an element of that kernel, we can separate it as a sum of an element of $\mathrm{Ker}(\partial_t)$ and an element of $\mathrm{Im}(t)$. Applying $\partial_t$, we find a non-trivial element in the kernel of $\partial_t t=t\partial_t+1$, which is a contradiction with the above.
\end{proof}

\noindent \textit{Proof of Proposition \ref{prop:Weyl}}. 

We have just established that the Weyl algebra action on $(\mathfrak{g}_Q)_{\mathbf{d}}$ is free with the generators $\mathrm{Ker}(\partial_t)_{\mathbf{d}}$. This implies that for each $\mathbf{d}$
 \[
\sum_{n\ge 0} \dim((\mathfrak{g}_Q)_{\mathbf{d},n})q^{-\frac12n}=\frac{1}{1-q^{-1}}\sum_{n\ge 0} \dim(\mathrm{Ker}(\partial_t)_{\mathbf{d},n})q^{-\frac12n}.
 \]
Since  $(\mathfrak{g}_Q^\vee)_{\mathbf{d},n}\simeq (\mathfrak{g}_Q)^\vee_{\mathbf{d},-n}$, we see that
 \[
(1-q)\mathrm{ch}((\mathfrak{g}_Q^\vee)_{\mathbf{d}},q)\in\mathbb{Z}_{\ge 0}[[q^{\frac12}]],
 \]
as required. \qed

To obtain further information about characters, we shall use Gr\"obner--Shirshov bases for Lie algebras. We shall use the ordering of generators ``dual'' to the one considered in the case of the algebra $\mathcal{A}_Q$. Namely, we first order the set of generators by choosing some ordering of $Q_0$, and then letting $b_{i,k}<b_{j,l}$ if $k>l$ or if $k=l$ and $i>j$; this ordering leads to the corresponding graded lexicographic ordering of monomials. We note that this order is not a well-order, but our relations are homogeneous with respect to the homological degree, and in each degree the number of monomials is finite, so all results of the theory of Gr\"obner--Shirshov bases are in fact available. Let us denote by $G_Q$ the reduced Gr\"obner basis of relations of $U(\mathfrak{g}_Q)$ for this ordering. 

It is immediate from the defining relations that the endomorphism $\tau$ of the free associative algebra with generators $b_{i,k}$ defined on generators by the formula $\tau(b_{i,k})=b_{i,k+1}$ descends to a well-defined endomorphism of $U(\mathfrak{g}_Q)$. Let us record an obvious but important feature of our ordering that would allow us to generalise the combinatorial argument of \cite{MR2889742} to arbitrary quivers.

\begin{prop}
Suppose that $m_1$ and $m_2$ are two monomials in the free associative algebra with generators $b_{i,k}$. Then $m_1<m_2$ if and only if $\tau(m_1)<\tau(m_2)$. Moreover, a monomial $m$ is a super-Lyndon--Shirshov word if and only if the monomial $\tau(m)$ is a super-Lyndon--Shirshov word.   
\end{prop}

\begin{prop}\label{prop:Lyndon}
For each $\mathbf{d}\in\mathbb{Z}_{\ge 0}^{Q_0}$, we have 
 \[
(1-q^{|\mathbf{d}|})\mathrm{ch}((\mathfrak{g}_Q^\vee)_\mathbf{d},q)\in \mathbb{Z}_{\ge 0}[q^{\frac12}].
 \]
\end{prop}

\begin{proof}
Let us denote by $N_Q$ the set of all super-Lyndon--Shirshov words which are normal with respect to $G_Q$, and by $N^0_Q$ the set of all super-Lyndon--Shirshov words which are normal with respect to $G_Q$ and the first letter is $b_{i,0}$ for some $i$. Note that for each such word $w=b_{i_1,k_1}\cdots b_{i_s,k_s}$ we have $k_1=\min(k_p)$, since otherwise there would be a cyclic shift which is larger than $w$. Thus, we have a bijection $w\mapsto (k_1,\tau^{-k_1}(w))$ between $N_Q$ and  $\mathbb{Z}_{\ge 0}\times N^0_Q$. Let us also note that the relations of $\mathfrak{g}_Q$ are all of the form $(z-w)^{a_{ij}}[b_i(z),b_j(w)]$ with certain $a_{ij}\ge 0$, meaning that the leading terms of $G_Q$ include the monomials $b_{i,k}b_{j,l}$ for all $i,j\in Q_0$ and all $l\ge k+a_{i,j}$, therefore for each $\mathbf{d}$ the set $(N_Q^0)_{\mathbf{d}}$ is finite. According to the theory of Gr\"obner--Shirshov bases for Lie algebras, the cosets of Lyndon--Shirshov monomials corresponding to elements of $N_Q$ form a basis of $\mathfrak{g}_Q$. Let us denote by $\mathfrak{g}_Q^0$ the span of the cosets of Lyndon--Shirshov monomials corresponding to elements of $N_Q^0$.  Since subtracting $k_1$ from each letter subtracts $2|\mathbf{d}|k_1$ from the homological degree, we see that 
 \[
\mathrm{ch}((\mathfrak{g}_Q^\vee)_\mathbf{d},q)=\frac{1}{1-q^{|\mathbf{d}|}}\mathrm{ch}(((\mathfrak{g}_Q^0)^\vee)_\mathbf{d},q),
 \] 
and $\mathrm{ch}(((\mathfrak{g}_Q^0)^\vee)_\mathbf{d},q)\in \mathbb{Z}_{\ge 0}[q^{\frac12}]$ since $(N_Q^0)_{\mathbf{d}}$ is finite.   
\end{proof}

We are now ready to establish an important property of the characters of our Lie algebras.

\begin{prop}\label{prop:Liepositivity}
For each $\mathbf{d}\in\mathbb{Z}_{\ge 0}^{Q_0}$, we have 
 \[
(1-q)\mathrm{ch}((\mathfrak{g}_Q^\vee)_{\mathbf{d}},q)\in\mathbb{Z}_{\ge 0}[q^{\frac12}].
 \]
More specifically, using the action of $\mathrm{Diff}_1$ on $\mathfrak{g}_Q$, we have
 \[
(1-q)\mathrm{ch}((\mathfrak{g}_Q^\vee)_{\mathbf{d}},q)=\sum_{n\ge 0} \dim(\mathrm{Ker}(\partial_t)_{\mathbf{d},n})q^{\frac12n}. 
 \]
\end{prop}

\begin{proof}
According to Proposition \ref{prop:Lyndon}, 
 \[
(1-q^{|\mathbf{d}|})\mathrm{ch}((\mathfrak{g}_Q^\vee)_{\mathbf{d}},q)\in \mathbb{Z}_{\ge 0}[q^{\frac12}].
 \]
At the same time, Proposition \ref{prop:Weyl} implies that 
 \[
(1-q^{|\mathbf{d}|})\mathrm{ch}((\mathfrak{g}_Q^\vee)_{\mathbf{d}},q)=\frac{1-q^{|\mathbf{d}|}}{1-q}(1-q)\mathrm{ch}((\mathfrak{g}_Q^\vee)_{\mathbf{d}},q)
 \]
is a product of a polynomial in $q^{\frac12}$ with non-negative coefficients and a power series in $q^{\frac12}$ with non-negative coefficients equal to 
 \[
\sum_{n\ge 0} \dim(\mathrm{Ker}(\partial_t)_{\mathbf{d},n})q^{\frac12n}.
 \]
This means that the latter power series must be a polynomial, concluding the proof.
\end{proof}

\subsection{Relationship to vertex Lie algebras}

We shall now see that the character $\mathrm{ch}((\mathfrak{g}_Q^\vee)_{\mathbf{d}},q)$ is directly related to the refined Donaldson--Thomas invariants of $Q$. To that end, we shall use the second key ingredient, that is, vertex Lie algebras (also known as Lie conformal algebras); the reader is invited to consult \cite[Sec. 4]{dotsenko_mozgovoy} and \cite{Ro} for details. Let us denote by $C_Q$ the free vertex Lie algebra corresponding to the locality function $N_Q(i,j)=m_{i,j}-\delta_{i,j}$; according to \cite[Sec. 5.1]{dotsenko_mozgovoy}, it is isomorphic to the free vertex Lie algebra corresponding to the non-negative locality function $N_Q^+(i,j)=\max(N_Q(i,j),0)$. Using the general theory of vertex Lie algebras, we may associate to the vertex Lie algebra $C_Q$ an honest Lie algebra $\mathfrak{L}_Q$, the \emph{coefficient algebra of $C_Q$}; moreover, we have a graded vector space decomposition 
 \[
\mathfrak{L}_Q=\mathfrak{L}_Q^-\oplus \mathfrak{L}_Q^+,
 \]
where $\mathfrak{L}_Q^-$ and $\mathfrak{L}_Q^+$ are Lie subalgebras of $\mathfrak{L}_Q$. It is established in \cite[Sec. 3]{Ro} that both the Lie algebra $\mathfrak{L}_Q$ and its subalgebra $\mathfrak{L}_Q^+$ admit explicit presentations by generators and relations as follows. The Lie algebra $\mathfrak{L}_Q$ is generated by elements $i(k)$ of degree $(\alpha_i,2k+m_{i,i}+1)$, $i\in Q_0$, $k\in\mathbb{Z}$, subject to the relations
\begin{equation}\label{eq:local}
\sum_{p=0}^{N_Q^+(i,j)} (-1)^p\binom{N_Q^+(i,j)}{p} [i(k-p),j(l+p)]=0 
\end{equation}
for all $i, j\in Q_0$, and the Lie algebra $\mathfrak{L}_Q^+$ is generated by elements $i(k)$ of degree $(\alpha_i,2k+m_{i,i}+1)$, $i\in Q_0$, $k\ge 0$, subject to those of the relations \eqref{eq:local} that only contain the generators $i(k)$ with $k\ge 0$. The subalgebra $\mathfrak{L}_Q^-$ is defined more indirectly. 
Examining the relations \eqref{eq:local}, we see that we have a Lie algebra isomorphism 
 \[
\mathfrak{L}_Q^+\cong\fg_Q 
 \]
sending $i(k)$ to $b_{i,k}$. We shall now use this observation to obtain a new interpretation of the motivic generating function. 

\begin{thm}\label{th:char-symmetry}
The Poincar\'e series $P(U(\fg_Q)^\vee,x,q)$ belongs to the subring 
 \[
\mathbb{Q}(q^{\frac12})[[x_i\colon i\in Q_0]]\subset\mathbb{Q}((q^{\frac12}))[[x_i\colon i\in Q_0]].
 \] 
In that subring, we have the equality 
 \[
A_Q(x,q)P(U(\fg_Q)^\vee,x,q)=1.
 \]
\end{thm}

\begin{proof}
We shall use the results of \cite[Sec. 6]{dotsenko_mozgovoy} that interpret the CoHA-modules $\M_\bw$, $\bw\in\mathbb{Z}_{\ge 0}^{Q_0}$ of \cite{franzen_chow,franzen_cohomology} in the context of vertex Lie algebras. For that, we shall consider the automorphism $\tau_\bw$ of the Lie algebra $\mathfrak{L}_Q$ defined on the generators as $\tau_\bw(i(n)):= i(n-\bw_i)$. 
This automorphism may be used to define two new Lie subalgebras $\mathfrak{L}_{Q,\bw}^-:=\tau_\bw(\mathfrak{L}_Q^-)$ and $\mathfrak{L}_{Q,\bw}^+:=\tau_\bw(\mathfrak{L}_Q^+)$. According to \cite[Th. 6.4]{dotsenko_mozgovoy}, we have the isomorphism
 \[
\M_\bw=(U(\mathfrak{L}_{Q,\bw}^+)\otimes_{U(\mathfrak{L}_{Q}^+)}\mathbb{Q})^\vee,
 \] 
so 
 \[
P(\M_\bw,x,q)=P((U(\mathfrak{L}_{Q,\bw}^+)\otimes_{U(\mathfrak{L}_{Q}^+)}\mathbb{Q})^\vee,x,q).
 \]
Note that using the isomorphism $\tau_\bw$, we can write 
 \[
U(\mathfrak{L}_{Q,\bw}^+)\otimes_{U(\mathfrak{L}_{Q}^+)}\mathbb{Q}=\tau_\bw(U(\mathfrak{L}_{Q}^+)\otimes_{U(\tau_\bw^{-1}(\mathfrak{L}_{Q}^+))}\mathbb{Q}) ,
 \]
so 
 \[
P((U(\mathfrak{L}_{Q,\bw}^+)\otimes_{U(\mathfrak{L}_{Q}^+)}\mathbb{Q})^\vee,x,q)=S_{-2\bw}P((U(\mathfrak{L}_{Q}^+)\otimes_{\tau_\bw^{-1}(U(\mathfrak{L}_{Q}^+))}\mathbb{Q})^\vee,x,q),
 \]
where $S_{\bw}(x^\bd)=q^{\frac12\bw\cdot\bd}x^\bd$. It follows that
 \[
P((U(\mathfrak{L}_{Q}^+)\otimes_{\tau_\bw^{-1}(U(\mathfrak{L}_{Q}^+))}\mathbb{Q})^\vee,x,q)=S_{2\bw}P((U(\mathfrak{L}_{Q,\bw}^+)\otimes_{\mathfrak{L}_{Q}^+}\mathbb{Q})^\vee,x,q)=S_{2\bw}P(\M_\bw,x,q).
 \]
Note that \emph{a priori} the left hand side of this equation is a power series with coefficients in $\mathbb{Q}[[q^{\frac12}]]$, and the right hand side is a power series with coefficients in $\mathbb{Q}[q^{\pm\frac12}]$, since each graded component of $\M_\bw$ is (degree shifted) cohomology of an algebraic variety, so in reality both sides have polynomials in $q^{\frac12}$ as coefficients. According to \cite[Prop. 3.4]{dotsenko_mozgovoy}, we have 
 \[
P(\M_\bw,x,q)=A_Q(x,q)\cdot S_{-2\bw}A_Q(x,q)^{-1} ,
 \]
so 
 \[
S_{2\bw}P(\M_\bw,x,q)=S_{2\bw}A_Q(x,q)\cdot A_Q(x,q)^{-1},
 \]
and to compare it with $P((U(\mathfrak{L}_{Q}^+)\otimes_{\tau_\bw^{-1}(U(\mathfrak{L}_{Q}^+))}\mathbb{Q})^\vee,x,q)$, we wish to expand it as a formal power series with coefficients in $\mathbb{Q}[[q^{\frac12}]]$, which we may do using the result of Proposition \ref{prop:poincare series} for the motivic generating function. This shows that $S_{2\bw}A_Q(x,q)$ has the limit $1$ as $\bw\to\infty$ (that is, all $\bw_i\to\infty$), so the limit of $S_{2\bw}P(\M_\bw,x,q)$ is $A_Q(x,q)^{-1}$. On the other hand, $P((U(\mathfrak{L}_{Q}^+)\otimes_{\tau_\bw^{-1}(U(\mathfrak{L}_{Q}^+))}\mathbb{Q})^\vee,x,q)$ has the limit $P(U(\mathfrak{L}_Q^+)^\vee,x,q)$ as $\bw\to\infty$ (since the graded components of the relative tensor product manifestly stabilize). Recalling that $\fg_Q\cong \mathfrak{L}_Q^+$, we see that 
 \[
P(U(\fg_Q)^\vee,x,q)=A_Q(x,q)^{-1} , 
 \] 
as required.  
\end{proof}

\subsection{A new proof of positivity of the refined DT invariants}
We are now ready to relate the Lie algebra $\fg_Q$ to the refined Donaldson--Thomas invariants of the quiver $Q$.

\begin{thm}\label{thm:koszul-and-dt}
The refined Donaldson--Thomas invariants of $Q$ can be computed using the $\mathrm{Diff}_1$-module structure on $\mathfrak{g}_Q$: for each $\mathbf{d}\in\mathbf{Z}_{\ge 0}^{Q_0}$, we have 
 \[
\mathrm{DT}_\mathbf{d}(q)=\sum_{n\ge 0} \dim(\mathrm{Ker}(\partial_t)_{\mathbf{d},n})q^{\frac12n-1}.
 \] 
In particular, $\mathrm{DT}_\mathbf{d}(q)\in\mathbb{Z}_{\ge 0}[q^{\pm\frac12}]$. 
\end{thm}

\begin{proof}
We have already established that the Lie algebra $\mathfrak{g}_Q$ is a free $\mathbb{Q}[t]$-module with the space of generators $\mathrm{Ker}(\partial_t)$. It remains to relate the refined Donaldson--Thomas invariants to the latter space of generators. Combining the definition 
$$A_Q(x,q^{-1})=\mathrm{Exp}\left(\frac{1}{1-q}\sum_{\mathbf{d}\in\mathbb{Z}_{\ge 0}^{Q_0}} (-1)^{\chi({\bf d},{\bf d})} {\rm DT}_{\bf d}(q^{-1})x^{\bf d}\right)$$
of the refined Donaldson--Thomas invariants with plethystic logarithm of the formula 
 \[
A_Q(x,q)P(U(\fg_Q)^\vee,x,q)=1
 \] 
of Theorem \ref{th:char-symmetry}, we conclude that 
 \[
\frac{(-1)^{\chi({\bf d},{\bf d})} {\rm DT}_{\bf d}(q)}{1-q^{-1}}+P((\mathfrak{g}_Q^\vee)_\mathbf{d},q)=0.
 \]
Let us recall that the homological degree of each generator $b_{i,k}$ of the Lie algebra $\fg_Q$ is equal to $2k+m_{i,i}+1\equiv m_{i,i}+1\pmod{2}$. Thus, all elements of $(\mathfrak{g}_Q)_{\mathbf{d}}$ (and of $(\mathfrak{g}_Q^\vee)_{\mathbf{d}}$) are in homological degree congruent to $\sum_{i\in Q_0}(m_{i,i}+1)\mathbf{d}_i$ modulo~$2$. However, 
 \[
\sum_{i\in Q_0}(m_{i,i}+1)\mathbf{d}_i\equiv \sum_{i\in Q_0}(1-m_{i,i})\mathbf{d}_i^2\equiv \chi(\mathbf{d},\mathbf{d})\pmod{2}, 
 \]
so we have
 \[
P((\mathfrak{g}_Q^\vee)_{\mathbf{d}},q)=(-1)^{\chi(\mathbf{d},\mathbf{d})}\mathrm{ch}((\mathfrak{g}_Q^\vee)_{\mathbf{d}}),
 \]
implying that
 \[
\frac{(-1)^{\chi({\bf d},{\bf d})} {\rm DT}_{\bf d}(q)}{1-q^{-1}}+(-1)^{\chi({\bf d},{\bf d})}\mathrm{ch}(\mathfrak{g}^\vee_Q)_\mathbf{d},q)=0.
 \]
This simplifies to 
 \[
{\rm DT}_{\bf d}(q)=q^{-1}(1-q)\mathrm{ch}(\mathfrak{g}^\vee_Q)_\mathbf{d},q), 
 \]
and it remains to apply Proposition \ref{prop:Liepositivity} to conclude that 
 \[
{\rm DT}_{\bf d}(q)=q^{-1}\sum_{n\ge 0} \dim(\mathrm{Ker}(\partial_t)_{\mathbf{d},n})q^{\frac12n},
 \]
as required.
\end{proof}

\section{The Koszulness conjecture}\label{koszulness}
We begin this section with recording the following theorem which is an immediate consequence of the results obtained in the previous sections. 

\begin{thm}
The algebra $\mathcal{A}_Q$ is numerically Koszul for every symmetric quiver $Q$.
\end{thm}

\begin{proof}
According to Theorem \ref{th:char-symmetry}, we have
 \[
A_Q(x,q)P(U(\fg_Q)^\vee,x,q)=1. 
 \]
Since $\mathcal{A}_Q^!\cong U(\fg_Q)$, the result of Proposition \ref{prop:poincare series} implies that
 \[
P(\mathcal{A}_Q,q^{\frac12}x,q)P((\mathcal{A}^!)^\vee,x,q)=1,
 \]
which, up to a change of variables, is precisely the numerical Koszulness of $\mathcal{A}_Q$.
\end{proof}

This result may be interpreted as strong evidence for the following conjecture. 

\begin{conj}\label{conj:Koszul}
The algebra $\mathcal{A}_Q$ is Koszul for every symmetric quiver $Q$. 
\end{conj}

In this section, we survey the situations in which we managed to prove this conjecture. 

\subsection{Koszulness via quadratic Gr\"obner bases}\label{sec:GB}

Let us classify, for a certain ordering of monomials, all quivers $Q$ such that the algebra $\mathcal{A}_Q$ has a noncommutative quadratic Gr\"obner basis. Let us choose some ordering of $Q_0$, and order the generators of $\mathcal{A}_Q$ as follows: $a_{i,k}<a_{j,l}$ if $k<l$ or if $k=l$ and $i<j$. This ordering of generators gives rise to the corresponding graded lexicographic ordering of monomials: to compare two monomials, we first compare their lengths, and if the lengths coincide, compare them lexicographically, letter by letter.  

\begin{definition}
Let $N$ be a positive integer. We say that a quiver $Q$ is almost $N$-regular if for all $i\ne j\in Q_0$, we have $m_{ij}=N$, and for each $i\in Q_0$ we have either $m_{ii}=N$ or $m_{ii}=N+1$.
\end{definition}

\begin{thm}\label{Nregular}
Let $Q$ be a symmetric quiver. For the ordering defined above, the algebra $\mathcal{A}_Q$ has a noncommutative quadratic Gr\"obner basis of relations if and only if the quiver $Q$ is almost $N$-regular. In particular, for each almost $N$-regular quiver $Q$, the algebra $\mathcal{A}_Q$ is Koszul.
\end{thm}

The proof of this theorem occupies the rest of the section. We first prepare several lemmas.

One can easily see that among the leading terms of the relations of the algebra $\mathcal{A}_Q$ there are, in general, repetitions. To present a candidate for a Gr\"obner basis, one has to begin with replacing these relations by relations with pairwise distinct leading terms.

\begin{lem}\label{lm:quadLT}
The ideal of relations of the algebra $\mathcal{A}_Q$ has a system of generators whose set of pairwise distinct leading terms is 
\begin{gather*}
a_{i,k}a_{j,l}, \quad k>l,\\
a_{i,k}a_{j,l}, \quad 0\le l-k\le m_{i,j}, i<j,\\
a_{i,k}a_{j,l}, \quad 0\le l-k\le m_{i,j}-1, i\ge j.
\end{gather*}
\end{lem}

\begin{proof}
Let us analyse the relations carefully. First of all, the first two groups of relations (supercommutativity) imply that $a_{i,k}a_{j,l}$ with $k>l$ are leading terms, as well as $a_{i,k}a_{j,k}$ with $i>j$, and $a_{i,k}a_{i,k}$, the latter in case of odd $m_{i,i}$. In what follows, we shall only consider monomials that are already normal with respect to these leading terms.

From Lemma \ref{lem:Gauss} (proved in Appendix), it follows that after the appropriate row reduction
\begin{itemize}
\item if $m_{i,i}$ is even, then $a_{i,k}^2$, $a_{i,k} a_{i,k+1}$, \ldots, $a_{i,k}a_{i,k+m_{i,i}-1}$ are the leading terms of the third group of relations,
\item if $m_{i,i}$ is odd, then $a_{i,k} a_{i,k+1}$, \ldots, $a_{i,k}a_{i,k+m_{i,i}-1}$ are the leading terms of the third group of relations,
\item if $i<j$, then 
$a_{i,k}a_{j,k}$, $a_{i,k} a_{j,k+1}$, $a_{i,k} a_{j,k+m_{i,j}}$ as well as $a_{j,k}a_{i,k+1}$, \ldots,  $a_{j,k} a_{i,k+m_{i,j}-1}$,  are the leading terms of the third group of relations,
\end{itemize}

It remains to notice that the union of these sets of leading terms is precisely the set described above.
\end{proof}

We already mentioned above that a generating set $G$ of $I$ is a Gr\"obner basis of $I$ if and only if cosets of monomials that are normal with respect to $G$ form a basis of the quotient modulo $I$. However, if all elements of $G$ are of weight two, there is a much more efficient criterion using the Diamond Lemma \cite[Th.~2.4.1.5]{BD}: a generating set $G$ of weight two is a Gr\"obner basis if and only if cosets of \emph{cubic} (weight three) monomials that are normal with respect to $G$ form a basis of the cubic part of the quotient modulo $I$. We shall now use this criterion in our case. In fact, for the algebra $\mathcal{A}_Q$, the cubic part can be further separated according to the $\mathbb{Z}^{Q_0}$-grading. Let us consider various possible situations.

\begin{lem}
The $q$-character of $\mathcal{A}_{Q,3\alpha_i}$ is equal to the character of the set of monomials  
\[
a_{i,k_1} a_{i,k_2} a_{i,k_3}, \ k_p\ge 0,   
\]
that are normal with respect to the quadratic leading terms of Lemma \ref{lm:quadLT}.
\end{lem}

\begin{proof}
From Proposition \ref{prop:functional-realisation}, we know that the graded component $\mathcal{A}_{Q,3\alpha_i}^\vee$ is identified with the vector space of polynomials equal to a product of a symmetric polynomial in $z_1,z_2,z_3$ and the polynomial $((z_1-z_2)(z_2-z_3)(z_1-z_3))^{m_{i,i}}$ shifted by $3m_{i,i}$. The $q$-character of this latter vector space is $\frac{q^{6m_{ii}}}{(q)_3}$. The normal monomials are the monomials $a_{i,k_1} a_{i,k_2} a_{i,k_3}$ with $k_2-k_1\ge m_{i,i}$ and $k_3-k_2\ge m_{i,i}$. We see that
 \[
(k_1,k_2,k_3)=(m_{i,i},2m_{i,i},3m_{i,i})+(l_1,l_2,l_3),
 \] 
where $(l_1,l_2,l_3)$ is a partition into at most three parts, and so the result follows.  
\end{proof}

\begin{lem}\label{lm:diag1}
Suppose that $i<j$. The $q$-character of $\mathcal{A}_{Q,2\alpha_i+\alpha_j}$ is equal to the character of the set of monomials  
\[
a_{i_1,k_1} a_{i_2,k_2} a_{i_3,k_3},\quad  \alpha_{i_1}+\alpha_{i_2}+\alpha_{i_3}=2\alpha_i+\alpha_j\ k_p\ge 0,   
\]
that are normal with respect to the quadratic leading terms of Lemma \ref{lm:quadLT} if and only if $m_{i,i}=m_{i,j}$ or $m_{i,i}=m_{i,j}+1$.
\end{lem}

\begin{proof}
From Proposition \ref{prop:functional-realisation}, we know that $\mathcal{A}_{Q,2\alpha_i+\alpha_j}^\vee$ is identified with the vector space of polynomials equal to a product of a polynomial in $z_1,z_2,z_3$ symmetric in $z_1,z_2$ and the polynomial $(z_1-z_2)^{m_{i,i}}((z_2-z_3)(z_1-z_3))^{m_{i,j}}$ shifted by $2m_{i,i}+m_{j,j}$. The $q$-character of this vector space is $\frac{q^{3m_{i,i}+2m_{i,j}+m_{j,j}}}{(q)_2(q)_1}$. There are three types of normal monomials: 
\begin{itemize}
\item $a_{i,k_1} a_{i,k_2} a_{j,k_3}$ with $k_2-k_1\ge m_{ii}$ and $k_3-k_2\ge m_{ij}+1$, contributing $\frac{q^{3m_{i,i}+m_{i,j}+m_{j,j}+1}}{(q)_3}$ to the character, 
\item $a_{i,k_1} a_{j,k_2} a_{i,k_3}$ with $k_2-k_1\ge m_{ij}+1$ and $k_3-k_2\ge m_{ij}$, contributing $\frac{q^{2m_{i,i}+3m_{i,j}+m_{j,j}+2}}{(q)_3}$ to the character, 
\item $a_{j,k_1} a_{i,k_2} a_{i,k_3}$ with $k_2-k_1\ge m_{ij}$ and $k_3-k_2\ge m_{ii}$, contributing $\frac{q^{3m_{i,i}+2m_{i,j}+m_{j,j}}}{(q)_3}$ to the character. 
\end{itemize}
Dividing by $q^{2m_{i,i}+m_{j,j}}$, we see that it remains to determine when we have the equality 
 \[
\frac{q^{2m_{i,i}+m_{i,j}+1}}{(q)_3}+\frac{q^{3m_{i,j}+2}}{(q)_3}+\frac{q^{2m_{i,j}+m_{i,i}}}{(q)_3}=\frac{q^{m_{i,i}+2m_{ij}}}{(q)_2(q)_1}, 
 \]
or, equivalently, 
 \[
q^{2m_{i,i}+m_{i,j}+1}+q^{3m_{i,j}+2}+q^{2m_{i,j}+m_{i,i}}=q^{m_{i,i}+2m_{i,j}}(1+q+q^2). 
 \]
This latter simplifies to  
 \[
q^{2m_{i,i}+m_{i,j}+1}+q^{3m_{i,j}+2}=q^{m_{i,i}+2m_{i,j}+1}+q^{m_{i,i}+2m_{i,j}+2},  
 \]
or
 \[
q^{m_{i,j}+1}(q^{m_{i,j}}-q^{m_{i,i}})(q^{m_{i,j}+1}-q^{m_{i,i}})=0,
 \]
and the claim follows.
\end{proof}

\begin{lem}\label{lm:diag2}
Suppose that $i>j$. The $q$-character of $\mathcal{A}_{Q,2\alpha_i+\alpha_j}$ is equal to the character of the set of monomials  
\[
a_{i_1,k_1} a_{i_2,k_2} a_{i_3,k_3},\quad  \alpha_{i_1}+\alpha_{i_2}+\alpha_{i_3}=2\alpha_i+\alpha_j\ k_p\ge 0,   
\]
that are normal with respect to the quadratic leading terms of Lemma \ref{lm:quadLT} if and only if $m_{i,i}=m_{i,j}$ or $m_{i,i}=m_{i,j}+1$.
\end{lem}

\begin{proof}
From Proposition \ref{prop:functional-realisation}, we know that $\mathcal{A}_{Q,2\alpha_i+\alpha_j}^\vee$ is identified with the vector space of polynomials equal to a product of a polynomial in $z_1,z_2,z_3$ symmetric in $z_1,z_2$ and the polynomial $(z_1-z_2)^{m_{i,i}}((z_2-z_3)(z_1-z_3))^{m_{i,j}}$ shifted by $2m_{i,i}+m_{j,j}$. The $q$-character of this vector space is $\frac{q^{3m_{i,i}+2m_{i,j}+m_{j,j}}}{(q)_2(q)_1}$. There are three types of normal monomials: 
\begin{itemize}
\item $a_{i,k_1} a_{i,k_2} a_{j,k_3}$ with $k_2-k_1\ge m_{i,i}$ and $k_3-k_2\ge m_{i,j}$, contributing $\frac{q^{3m_{i,i}+m_{i,j}+m_{j,j}}}{(q)_3}$ to the character, 
\item $a_{i,k_1} a_{j,k_2} a_{i,k_3}$ with $k_2-k_1\ge m_{i,j}$ and $k_3-k_2\ge m_{i,j}+1$, contributing $\frac{q^{2m_{i,i}+3m_{i,j}+m_{j,j}+1}}{(q)_3}$ to the character, 
\item $a_{j,k_1} a_{i,k_2} a_{i,k_3}$ with $k_2-k_1\ge m_{i,j}+1$ and $k_3-k_2\ge m_{i,i}$, contributing $\frac{q^{3m_{i,i}+2m_{i,j}+m_{j,j}+2}}{(q)_3}$ to the character. 
\end{itemize}
Dividing by $q^{2m_{i,i}+m_{j,j}}$, we see that it remains to determine when we have the equality 
 \[
\frac{q^{2m_{i,i}+m_{i,j}}}{(q)_3}+\frac{q^{3m_{i,j}+1}}{(q)_3}+\frac{q^{2m_{i,j}+m_{i,i}+2}}{(q)_3}=\frac{q^{m_{i,i}+2m_{i,j}}}{(q)_2(q)_1}, 
 \]
or, equivalently, 
 \[
q^{2m_{i,i}+m_{i,j}}+q^{3m_{i,j}+1}+q^{2m_{i,j}+m_{i,i}+2}=q^{m_{i,i}+2m_{i,j}}(1+q+q^2). 
 \]
This latter simplifies to  
 \[
q^{2m_{i,i}+m_{i,j}}+q^{3m_{i,j}+1}=q^{m_{i,i}+2m_{i,j}}+q^{m_{i,i}+2m_{i,j}+1},  
 \]
or
 \[
q^{m_{i,j}}(q^{m_{i,j}}-q^{m_{i,i}})(q^{m_{i,j}+1}-q^{m_{i,i}})=0,
 \]
and the claim follows.
\end{proof}

\begin{lem}\label{lm:off-diag}
Suppose that $i<j<k$. The $q$-character of $\mathcal{A}_{Q,\alpha_i+\alpha_j+\alpha_k}$ is equal to the character of the set of monomials  
\[
a_{i_1,k_1} a_{i_2,k_2} a_{i_3,k_3},\quad  \alpha_{i_1}+\alpha_{i_2}+\alpha_{i_3}=\alpha_i+\alpha_j+\alpha_k,\ k_p\ge 0,   
\]
that are normal with respect to the quadratic leading terms of Lemma \ref{lm:quadLT} if and only if $m_{i,j}=m_{j,k}=m_{i,k}$.
\end{lem}

\begin{proof}
From Proposition \ref{prop:functional-realisation}, we know that $\mathcal{A}_{Q,\alpha_i+\alpha_j+\alpha_k}^\vee$ is identified with the vector space of polynomials equal to a product of a polynomial in $z_1,z_2,z_3$ and the polynomial $(z_1-z_2)^{m_{i,j}}(z_2-z_3)^{m_{j,k}}(z_1-z_3)^{m_{i,k}}$ shifted by $m_{i,i}+m_{j,j}+m_{k,k}$. The $q$-character of this vector space is $\frac{q^{m_{i,i}+m_{j,j}+m_{k,k}+m_{i,j}+m_{j,k}+m_{i,k}}}{(q)_1^3}$. There are six types of normal monomials: 
\begin{itemize}
\item $a_{i,k_1} a_{j,k_2} a_{k,k_3}$ with $k_2-k_1\ge m_{i,j}+1$ and $k_3-k_2\ge m_{j,k}+1$, contributing $\frac{q^{m_{i,i}+m_{j,j}+m_{k,k}+2m_{i,j}+m_{j,k}+3}}{(q)_3}$ to the character, 
\item $a_{i,k_1} a_{k,k_2} a_{j,k_3}$ with $k_2-k_1\ge m_{i,k}+1$ and $k_3-k_2\ge m_{j,k}$, contributing $\frac{q^{m_{i,i}+m_{j,j}+m_{k,k}+2m_{i,k}+m_{j,k}+2}}{(q)_3}$ to the character, 
\item $a_{j,k_1} a_{i,k_2} a_{k,k_3}$ with $k_2-k_1\ge m_{i,j}$ and $k_3-k_2\ge m_{i,k}+1$, contributing $\frac{q^{m_{i,i}+m_{j,j}+m_{k,k}+2m_{i,j}+m_{i,k}+1}}{(q)_3}$ to the character, 
\item $a_{j,k_1} a_{k,k_2} a_{i,k_3}$ with $k_2-k_1\ge m_{j,k}+1$ and $k_3-k_2\ge m_{i,k}$, contributing $\frac{q^{m_{i,i}+m_{j,j}+m_{k,k}+2m_{j,k}+m_{i,k}+2}}{(q)_3}$ to the character, 
\item $a_{k,k_1} a_{i,k_2} a_{j,k_3}$ with $k_2-k_1\ge m_{i,k}$ and $k_3-k_2\ge m_{i,j}+1$, contributing $\frac{q^{m_{i,i}+m_{j,j}+m_{k,k}+2m_{i,k}+m_{i,j}+1}}{(q)_3}$ to the character, 
\item $a_{k,k_1} a_{j,k_2} a_{i,k_3}$ with $k_2-k_1\ge m_{j,k}$ and $k_3-k_2\ge m_{i,j}$, contributing $\frac{q^{m_{i,i}+m_{j,j}+m_{k,k}+2m_{j,k}+m_{i,j}}}{(q)_3}$ to the character. 
\end{itemize}
Dividing by $q^{m_{i,i}+m_{j,j}+m_{k,k}}$, we see that it remains to determine when these contributions add up to $\frac{q^{m_{i,j}+m_{j,k}+m_{i,k}}}{(q)_1^3}$, or, in other words, when 
\begin{multline*}
q^{2m_{i,j}+m_{j,k}+3}+q^{2m_{i,k}+m_{j,k}+2}+q^{2m_{i,j}+m_{i,k}+1}+q^{2m_{j,k}+m_{i,k}+2}+\\
q^{2m_{i,k}+m_{i,j}+1}+q^{2m_{j,k}+m_{i,j}}=q^{m_{i,j}+m_{j,k}+m_{i,k}}(1+q)(1+q+q^2).
\end{multline*}
Dividing by $q^{m_{i,j}+m_{j,k}+m_{i,k}}$, one obtains an equivalent condition
 \[
q^{m_{i,j}-m_{i,k}+3}+q^{m_{i,k}-m_{i,j}+2}+q^{m_{i,j}-m_{j,k}+1}+q^{m_{j,k}-m_{i,j}+2}+
q^{m_{i,k}-m_{j,k}+1}+q^{m_{j,k}-m_{i,k}}=(1+q)(1+q+q^2). 
 \]
For the rest of the proof, let us denote $a:=m_{i,j}-m_{i,k}$, $b:=m_{i,k}-m_{j,k}$, so that our condition becomes 
 \[
q^{a+3}+q^{2-a}+q^{a+b+1}+q^{2-a-b}+q^{b+1}+q^{-b}=1+2q+2q^2+q^3.
 \]
Since $q$ is a formal variable, each term on the left must match a term on the right. Let us note that on the left we have both the exponent $-b$ and the exponent $b+1$, which add up to one. Examining the exponents on the right, we conclude that one of them is equal to zero and the other is equal to one. Also, we have the exponent $a+3$ and the exponent $2-a$ which add up to $5$, and examining the exponents on the right, we conclude that one of them is equal to $2$ and the other is equal to $3$. These two observations imply that $a,b\in\{-1,0\}$, and that $q^{a+b+1}+q^{2-a-b}=q+q^2$, so $a+b+1=1$ or $a+b+1=2$, or, in other words, $a+b\in\{0,1\}$, which, given that $a,b\in\{-1,0\}$, implies $a=b=0$, and the claim follows.   
\end{proof}

\noindent \textit{Proof of Theorem \ref{Nregular}}.
The cases we considered exhaust all possible $\mathbb{Z}^{Q_0}$-gradings of the cubic part, so the claim on the quadratic Gr\"obner basis follows: Lemma \ref{lm:off-diag} ensures that all off-diagonal elements are pairwise equal, and Lemmas \ref{lm:diag1} and \ref{lm:diag2} ensure that the diagonal elements are either equal to the corresponding off-diagonal ones or exceed them by one, so the conjunction of these (which constitutes the condition of our Gr\"obner basis criterion) singles out precisely the almost $N$-regular quivers. In conjunction with Proposition \ref{prop:KoszulGB}, this result implies that for each almost $N$-regular quiver $Q$, the algebra $\mathcal{A}_Q$ is Koszul. \qed
\medskip

The almost $N$-regularity condition interpolates between two cases each of which is meaningful in its own right. The condition $m_{ij}=N$ for all $i,j\in Q_0$ is reasonable from the quiver viewpoint: it says that the quiver is \emph{regular}, meaning that there is the same number of arrows between any two (possibly coinciding) vertices. The condition $m_{ij}=N$ for all $i\ne j\in Q_0$, $m_{ii}=N+1$ for all $i\in Q_0$ is reasonable from the Koszul dual viewpoint: in this case, the Koszul dual algebra turns out to be the free vertex algebra on $|Q_0|$ generators with the constant locality function $N(i,j)\equiv N$, see~\cite{Ro}. In fact, one can use results of \cite{Ro} (appropriately modified to eliminate certain misprints) to establish existence of a quadratic Gr\"obner basis for the ideal of relations of the Koszul dual algebra; a different approach to studying the same algebra which also leads to a quadratic Gr\"obner basis is developed in~\cite{Do}.

\begin{rem}
Note that the ordering that we consider is ``global'', forcing the quiver to be completely regular. In fact, it is easy to modify it slightly to handle connected components of $Q$ separately, and for such orderings the same argument proves that the algebra $\mathcal{A}_Q$ has a quadratic Gr\"obner basis if and only if each connected component of $Q$ is $N$-regular for some $N>0$. 
\end{rem}

\subsection{Koszulness beyond quadratic Gr\"obner bases}

In this section, we give example of a quiver for which the algebra $\mathcal{A}_Q$ does not have a quadratic Gr\"obner basis of relations for any admissible ordering but is nevertheless Koszul. We begin with a following proposition. 

\begin{prop}
Suppose that $Q$ is a quiver which has two vertices $i\ne j$ with exactly one arrow $i\to j$ and no loops at either $i$ or $j$. There exists no admissible ordering of monomials in the generators for which the algebra $\mathcal{A}_Q$ has a quadratic Gr\"obner basis of relations. 
\end{prop}

\begin{proof}
Let us consider the following defining relations of $\mathcal{A}_Q$:
\begin{gather*}
a_{i,0}a_{j,0}=0,\\
a_{i,0}a_{j,1}+a_{i,1}a_{j,0}=0 .
\end{gather*}
We remark that they are the only relations of the respective gradings. They immediately imply $a_{i,0}^2a_{j,1}=a_{i,1}a_{j,0}^2=0$. These elements of the ideal of relations of $\mathcal{A}_Q$ are monomials, and thus they must be divisible by the leading terms of some elements of the reduced Gr\"obner basis, no matter which admissible ordering we choose. Suppose that the algebra $\mathcal{A}_Q$ has a quadratic Gr\"obner basis. Since the relations of our algebra do not contain squares of variables, the same is true for the reduced Gr\"obner basis. Therefore $a_{i,0}a_{j,1}$ and $a_{i,1}a_{j,0}$ must both be leading terms of some elements of the reduced Gr\"obner basis, but there is just one relation $a_{i,0}a_{j,1}+a_{i,1}a_{j,0}$ of the corresponding degree, so we arrived at a contradiction. 
\end{proof}

\begin{prop}\label{prop:reg}
Let $Q$ be a quiver with two vertices $0$ and $1$ and exactly two arrows: one from $0$ to $1$ and the other from $1$ to $0$. The algebra $\mathcal{A}_Q$ is Koszul. 
\end{prop}

\begin{proof}
In this case, all generators commute, and the relations of the third type are the relations $a_0(z)a_1(z)=0$. Let us consider the algebra $\mathcal{A}_Q^{(p)}$ which has relations $a_0(z)a_1(z)=0$, but we truncate the ring of power series at $z^{p+1}$, so that there are only finitely many relations. We remark that the algebra $\mathcal{A}_Q$ is Koszul if and only if the algebras $\mathcal{A}_Q^{(p)}$ are Koszul for all $p$. Indeed, not only the algebra $\mathcal{A}_Q$ has its additional homological grading of $\mathcal{A}_Q$, but that same grading also exists on the minimal resolution of the ground field, and computing every graded part of the resolution involves dealing with finitely many of the generators and finitely many of the relations between them.  We shall demonstrate that the left-hand sides of the defining relations 
\begin{gather*}
a_{0,0}a_{1,0}=0,\\
a_{0,0}a_{1,1}+a_{0,1}a_{1,0}=0,\\
\ldots\\
a_{0,0}a_{1,p}+\cdots+a_{0,p}a_{1,0}=0
\end{gather*}
of the algebra $\mathcal{A}_Q^{(p)}$ form a regular sequence. 

Let us denote by $R$ the ring of polynomials in the variables $a_{0,k}$, $a_{1,k}$, and consider the obvious short exact sequence 
 \[
0\to a_{0,0}R/a_{0,0}a_{1,0}R\to R/a_{0,0}a_{1,0}R\to R/a_{0,0}R\to 0
 \]
of $R$-modules. Our goal is to show that the elements
\begin{gather*}
a_{0,0}a_{1,1}+a_{0,1}a_{1,0}=0,\\
\ldots\\
a_{0,0}a_{1,p}+\cdots+a_{0,p}a_{1,0}=0
\end{gather*}
form an $M$-regular sequence for the $R$-module $M= R/a_{0,0}a_{1,0}R$. Using the Koszul complex criterion for regular sequences \cite[Th.~16.5]{Ma}, it is easy to see that it is sufficient to prove the same for the modules $a_{0,0}R/a_{0,0}a_{1,0}R$ and $R/a_{0,0}R$. The action of $R$ on the first of these modules factors through the action of $R/a_{1,0}R$, and $a_{0,0}R/a_{0,0}a_{1,0}R$ is easily seen to be a free $R/a_{1,0}R$-module of rank one. The action of $R$ on the second module factors through the action of $R/a_{0,0}R$, and that module is manifestly a free $R/a_{0,0}R$-module of rank one. Setting $a_{0,0}=0$ in the $j$-th relation $a_{0,0}a_{1,j}+\cdots+a_{0,j}a_{1,0}=0$ produces 
the $(j-1)$-st relation of the same shape using the variables $\{a_{0,1+k}\mid k\ge 0\}$ instead of  $\{a_{0,k}\mid k\ge 0\}$. This allows us to conclude that the regularity of our sequence follows by induction on $p$.
\end{proof}

\begin{prop}
Let $Q$ be a connected symmetric quiver. The minimal set of relations of the algebra $\mathcal{A}_Q$, considered as a commutative algebra with generators $a_{i,k}$, forms a regular sequence if and only if $Q$ is one of the following quivers:
\begin{itemize}
\item the quiver with one vertex and no loops,
\item the quiver with one vertex and one loop,
\item the quiver with two vertices $0$ and $1$ and exactly two arrows: one from $0$ to $1$ and the other from $1$ to $0$. 
\end{itemize}
\end{prop}

\begin{proof}
Suppose first that $Q$ has just one vertex. If $Q$ has $2m\ge 2$ loops at its only vertex $0$, the minimal set of relations of the algebra $\mathcal{A}_Q$ includes  $a_{0,0}^2$ and $a_{0,0}a_{0,1}$, which prevents it from being regular. If $Q$ has $2m+1\ge 3$ loops at its only vertex $0$, the minimal set of relations of the algebra $\mathcal{A}_Q$ includes $a_{0,0}a_{0,1}$, which is not a regular element (it annihilates $a_{0,0}$). Thus, the only remaining cases are the quiver with no loops and the quiver with just one loop. In the case of no loops, there are no relations of the second group; in the case of one loop, the relations of the second group are redundant. Either way, the statement is trivially true. 
 
Suppose that $Q$ has more than one vertex. Since $Q$ is connected, it has two vertices $i\ne j$ with $m_{i,j}\ne 0$. If $m_{i,j}\ge 2$, the minimal set of relations of the algebra $\mathcal{A}_Q$ includes $a_{i,0}a_{j,0}$ and $a_{i,0}a_{j,1}$, which prevents it from being regular. If $m_{i,j}=1$ and one of the two vertices, say $i$, has at least one arrow to a vertex $k\ne j$ (note that we allow $k=i$, so it also accounts for possible loops), then the minimal set of relations of the algebra $\mathcal{A}_Q$ includes $a_{i,0}a_{j,0}$ and $a_{i,0}a_{k,0}$, which prevents it from being regular. Thus, the only remaining case is that of the quiver $Q$ with $Q_0=\{i,j\}$ and exactly two arrows, for which the statement is established in the proof of Proposition \ref{prop:reg}. 
\end{proof}

\appendix
\section{Reduction of relations of \texorpdfstring{$\mathcal{A}_Q$}{AQ}}

In this section, we state and prove a simple but slightly technical result which is used in the paper. 

\begin{lem}\label{lem:Gauss}\leavevmode
\begin{enumerate}
\item Let $a_i$, $i\ge0$, be a sequence of commuting formal variables. Let $a(z)=\sum_{i\ge 0} a_iz^i$, and consider, for the given $m\ge 1$, the system of elements $R$ defined as the coefficients of the series $a(z)^2$, $a(z)a'(z)$, \ldots, $a(z)a^{(2m-1)}(z)$. Let us order the quadratic monomials $a_ia_j$ with the given $k=i+j$ by saying that $a_ia_j>a_{i'}a_{j'}$ if $|i-j|<|i'-j'|$. Then there is a system of elements $R'$ spanning the same vector space as $R$ whose set of leading monomials is $\{a_ia_j \colon 0\le j-i\le 2m-1\}$.
\item Let $b_i$, $i\ge0$, be a sequence of anti-commuting formal variables. Let $b(z)=\sum_{i\ge 0} b_iz^i$, and consider, for the given $m\ge 2$, the system of elements $R$ defined as the coefficients of the series $b(z)^2$, $b(z)b'(z)$, \ldots, $b(z)b^{(2m-2)}(z)$. Let us order the quadratic monomials $b_ib_j$ with the given $k=i+j$ by saying that $b_ib_j>b_{i'}b_{j'}$ if $|i-j|<|i'-j'|$. Then there is a system of elements $R'$ spanning the same vector space as $R$ whose set of leading monomials is $\{b_ib_j\colon 0\le j-i\le 2m-2\}$.
\item Let $c_i$, $i\ge0$, and $d_i$, $i\ge 0$, be two sequence of formal variables. Let $c(z)=\sum_{i\ge 0} c_iz^i$ and $d(z)=\sum_{i\ge 0} d_iz^i$, and consider, for the given $m\ge 1$, the system of elements $R$ defined as the coefficients of the series $c(z)d(z)$, $c(z)d'(z)$, \ldots, $c(z)d^{(m-1)}(z)$. Let us order the quadratic monomials $c_id_j$ with the given $k=i+j$ by saying that $c_id_j>c_{i'}d_{j'}$ if $|i-j|<|i'-j'|$, or if  $|i-j|=|i'-j'|$ and $i<i'$. Then there is a system of elements $R'$ spanning the same vector space as $R$ whose set of leading monomials is 
 \[
\{c_i d_j\colon 0\le j-i\le m \}\cup \{c_i d_j\colon 0< i-j< m \}.
 \]
\end{enumerate}
\end{lem}

\begin{proof}
For commuting variables, it is easy to see that for each odd $l$ the relation $a(z)a^{(l)}(z)$ is redundant. Now let us fix $k\ge 0$, and examine the elements of $R$ obtained as the coefficients of $z^k$ of the series $a(z)^2$, $\frac12a(z)a^{(2)}(z)$\ldots, $\frac{1}{(2m-2)!}a(z)a^{(2m-2)}(z)$. If $k=2p$, the matrix relating these elements to the monomial basis is
 \[
\begin{pmatrix}
1&2&\cdots& 2&2\\
\binom{p}{2}& \binom{p-1}{2}+\binom{p+1}{2}&\cdots & \binom{1}{2}+\binom{2p-1}{2}& \binom{0}{2}+\binom{2p}{2}\\
\binom{p}{4}& \binom{p-1}{4}+\binom{p+1}{4}&\cdots & \binom{1}{4}+\binom{2p-1}{4}& \binom{0}{4}+\binom{2p}{4}\\
\vdots&\vdots&\ddots&\vdots&\vdots\\
\binom{p}{2m-2}& \binom{p-1}{2m-2}+\binom{p+1}{2m-2}&\cdots & \binom{1}{2m-2}+\binom{2p-1}{2m-2}& \binom{0}{2m-2}+\binom{2p}{2m-2}
\end{pmatrix}.
 \]
If we subtract from each column the previous one (and from the second column twice the previous one), we obtain the matrix 
 \[
\begin{pmatrix}
1&0&\cdots& 0&0\\
\binom{p}{2}& \binom{p}{1}-\binom{p-1}{1}&\cdots & \binom{2p-2}{1}-\binom{1}{1}& \binom{2p-1}{1}-\binom{0}{1}\\
\binom{p}{4}& \binom{p}{3}-\binom{p-1}{3}&\cdots & \binom{2p-2}{3}-\binom{1}{3}& \binom{2p-1}{3}-\binom{0}{3}\\
\vdots&\vdots&\ddots&\vdots&\vdots\\
\binom{p}{2m-2}& \binom{p}{2m-3}-\binom{p-1}{2m-3}&\cdots & \binom{2p-1}{2m-3}-\binom{1}{2m-3}& \binom{2p}{2m-3}-\binom{0}{2m-3}
\end{pmatrix}, 
 \]
which we shall denote $A_{2p,m}$ and keep aside for now. If $k=2p+1$, the matrix relating these elements to the monomial basis is
 \[
\begin{pmatrix}
2&2&\cdots& 2&2\\
\binom{p}{2}+\binom{p+1}{2}& \binom{p-1}{2}+\binom{p+2}{2}& \cdots & \binom{1}{2}+\binom{2p}{2}& \binom{0}{2}+\binom{2p+1}{2}\\
\binom{p}{4}+\binom{p+1}{4}& \binom{p-1}{4}+\binom{p+2}{4}& \cdots & \binom{1}{4}+\binom{2p}{4}& \binom{0}{4}+\binom{2p+1}{4}\\
\vdots&\vdots&\ddots&\vdots&\vdots\\
\binom{p}{2m-2}+\binom{p+1}{2m-2}& \binom{p-1}{2m-2}+\binom{p+2}{2m-2}&\cdots & \binom{1}{2m-2}+\binom{2p}{2m-2}& \binom{0}{2m-2}+\binom{2p+1}{2m-2}
\end{pmatrix}.
 \]
If we subtract from each column the previous one (and from the second column twice the previous one), we obtain the matrix 
 \[
\begin{pmatrix}
2&0&\cdots& 0&0\\
\binom{p}{2}+\binom{p+1}{2}& \binom{p+1}{1}-\binom{p-1}{1}&\cdots & \binom{2p-1}{1}-\binom{1}{1}& \binom{2p}{1}-\binom{0}{1}\\
\binom{p}{4}+\binom{p+1}{4}& \binom{p+1}{3}-\binom{p-1}{3}&\cdots & \binom{2p-1}{3}-\binom{1}{3}& \binom{2p}{3}-\binom{0}{3}\\
\vdots&\vdots&\ddots&\vdots&\vdots\\
\binom{p}{2m-2}+\binom{p+1}{2m-2}& \binom{p+1}{2m-3}-\binom{p-1}{2m-3}&\cdots & \binom{2p-1}{2m-3}-\binom{1}{2m-3}& \binom{2p}{2m-3}-\binom{0}{2m-3}
\end{pmatrix},
 \]
which we shall denote $A_{2p+1,m}$ and keep aside for now.

For anti-commuting variables, it is easy to see that for each even $l$ the relation $b(z)b^{(l)}(z)$ is redundant. Now let us fix $k\ge 0$, and examine the elements of $R$ obtained as the coefficients of $z^k$ of the series $b(z)b'(z)$, $\frac1{3!}b(z)b^{(3)}(z)$\ldots, $\frac{1}{(2m-3)!}b(z)b^{(2m-3)}(z)$. If $k=2p$, the matrix relating these elements to the monomial basis is
 \[
\begin{pmatrix}
\binom{p+1}{1}-\binom{p-1}{1}& \binom{p+2}{1}-\binom{p-2}{1}&\cdots & \binom{2p-1}{2}-\binom{1}{1}& \binom{2p}{1}-\binom{0}{1}\\
\binom{p+1}{3}-\binom{p-1}{3}& \binom{p+2}{3}-\binom{p-2}{3}& \cdots & \binom{2p-1}{3}-\binom{1}{3}& \binom{2p}{3}-\binom{0}{3}\\
\vdots&\vdots&\ddots&\vdots&\vdots\\
\binom{p-1}{2m-3}-\binom{p-1}{2m-3}& \binom{p+2}{2m-3}-\binom{p-2}{2m-3}&\cdots & \binom{2p-1}{2m-3}-\binom{1}{2m-3}& \binom{2p}{2m-3}-\binom{0}{2m-3}
\end{pmatrix},
 \]
which is obtained from the matrix $A_{2p+1,m}$ above by deleting the first row and the first column and, using the formula $-\binom{p-1}{s}+\binom{p+1}{s}=\binom{p-1}{s-1}+\binom{p}{s-1}$, is seen to be column equivalent to the matrix $A_{2p-1,m-1}$ above.

Similarly, if $k=2p+1$, the matrix relating these elements to the monomial basis is
 \[
\begin{pmatrix}
\binom{p+1}{1}-\binom{p}{1}& \binom{p+2}{1}-\binom{p-1}{1}& \cdots & \binom{2p}{2}-\binom{1}{1}& \binom{2p+1}{1}-\binom{0}{1}\\
\binom{p+1}{3}-\binom{p}{3}& \binom{p+2}{3}-\binom{p-1}{3}&\cdots & \binom{2p}{3}-\binom{1}{3}& \binom{2p+1}{3}-\binom{0}{3}\\
\vdots&\vdots&\ddots&\vdots&\vdots\\
\binom{p-1}{2m-3}-\binom{p}{2m-3}& \binom{p+2}{2m-3}-\binom{p-1}{2m-3}&\cdots & \binom{2p}{2m-3}-\binom{1}{2m-3}& \binom{2p+1}{2m-3}-\binom{0}{2m-3}
\end{pmatrix},
 \]
which is obtained from the matrix $A_{2p,m}$ above by deleting the first row and the first column and, using the formula $-\binom{p}{s}+\binom{p+1}{s}=\binom{p}{s-1}$, is seen to be column equivalent to $A_{2p-2,m-1}$ above.

Overall, these calculations allow us to conclude by induction that our matrices are of full rank, and deduce the existence of the spanning sets of their row spaces of requisite form. \\

For the last statement, the proof is similar but simpler. If, for the given $k\ge 0$, we examine the elements of $R$ obtained as the coefficients of $z^k$ of the series $c(z)d(z)$, $c(z)d'(z)$, $\frac12c(z)d^{(2)}(z)$\ldots, $\frac{1}{(m-1)!}c(z)d^{(m-1)}(z)$, then for $k=2p$ the matrix relating these elements to the monomial basis is
 \[
\begin{pmatrix}
1&1&1&\cdots& 1&1\\
\binom{p}{1}& \binom{p+1}{1}& \binom{p-1}{1}& \cdots & \binom{2p}{1}&\binom{0}{1}\\
\binom{p}{2}& \binom{p+1}{2}& \binom{p-1}{2}& \cdots & \binom{2p}{2}&\binom{0}{2}\\
\vdots&\vdots&\vdots&\ddots&\vdots&\vdots\\
\binom{p}{m-1}& \binom{p+1}{m-1}& \binom{p-1}{m-1}& \cdots &  \binom{2p}{m-1}&\binom{0}{m-1}\\
\end{pmatrix},
 \]
and for $k=2p+1$ the matrix relating these elements to the monomial basis is  
 \[
\begin{pmatrix}
1&1&1&1&\cdots&1&1\\
\binom{p+1}{1}& \binom{p}{1}& \binom{p+2}{1}& \binom{p-1}{1}& \cdots & \binom{2p+1}{1}&\binom{0}{1}\\
\binom{p+1}{2}& \binom{p}{2}& \binom{p+2}{2}& \binom{p-1}{2}& \cdots & \binom{2p+1}{2}&\binom{0}{2}\\
\vdots&\vdots&\vdots&\vdots&\ddots&\vdots&\vdots\\
\binom{p+1}{m-1}& \binom{p}{m-1}& \binom{p+2}{m-1}& \binom{p-1}{m-1}& \cdots & \binom{2p+1}{m-1}&\binom{0}{m-1}\\
\end{pmatrix}.
 \]
Elementary row and column operations on these matrices easily show that these matrices are of full rank, and that their row spaces have spanning sets of requisite form. 
\end{proof}

\bibliographystyle{plain}
\bibliography{biblio}

\end{document}